\documentclass[11pt]{amsart}

\usepackage{amsmath,amsthm,amssymb,amsfonts,ifpdf}
\usepackage{xcolor}

\usepackage{xargs}  
\usepackage[colorinlistoftodos,prependcaption,textsize=tiny,obeyFinal]{todonotes}
\newcommandx{\ebltodo}[2][1=]{\todo[linecolor=red,backgroundcolor=red!25,bordercolor=red,#1]{#2}}
{
  \color{olive}%
}%
{}

\ifpdf
\usepackage[pdftex,pdfstartview=FitH,pdfborderstyle={/S/B/W 1},%
colorlinks=true, linkcolor=blue, urlcolor=blue, citecolor=blue,%
pagebackref=true]{hyperref}
\else
\usepackage[hypertex]{hyperref}
\fi
\usepackage{enumitem}
\setenumerate{leftmargin=*}
\setitemize{leftmargin=*}
\usepackage{amscd}
\usepackage[latin1]{inputenc}
\DeclareMathAlphabet{\mathpzc}{OT1}{pzc}{m}{it}
\usepackage{bbm}

\numberwithin{equation}{section}

\swapnumbers
\newtheorem{thm}[subsection]{Theorem}
\newtheorem{coro}[subsection]{Corollary}
\newtheorem*{cor*}{Corollary}
\newtheorem{lemma}[subsection]{Lemma}
\newtheorem*{sublemma}{Sublemma}

\newtheorem{propos}[subsection]{Proposition}

\newtheorem*{thm*}{Theorem}
\newtheorem*{thma*}{Theorem A}
\newtheorem*{thmb*}{Theorem B}
\newtheorem*{thmc*}{Theorem C}

\swapnumbers
\theoremstyle{definition}

\newtheorem{definition}[subsection]{Definition}

\newcounter{consta}

\newcounter{constk}

\newcounter{constc}

\newcounter{constE}

\newcounter{constd}

\makeatletter
\newcommand*\bigcdot{\mathpalette\bigcdot@{.5}}
\newcommand*\bigcdot@[2]{\mathbin{\vcenter{\hbox{\scalebox{#2}{$\m@th#1\bullet$}}}}}
\makeatother

\def\XXint#1#2#3{{\setbox0=\hbox{$#1{#2#3}{\int}$ }
\vcenter{\hbox{$#2#3$ }}\kern-.6\wd0}}

\DeclareMathOperator{\Lip}{Lip}
\DeclareMathOperator{\supp}{supp}

\DeclareMathOperator{\diff}{d}

\DeclareMathOperator\Mat{Mat}

\newcommand\vol{{\rm{vol}}}
\newcommand\SL{{\rm{SL}}}

\newcommand\GL{{\rm{GL}}}
\newcommand\PGL{{\rm{PGL}}}
\newcommand\SO{{\rm{SO}}}
\newcommand\Lie{{\rm Lie}}

\def\bbz{\mathbb{Z}}
\def\bbq{\mathbb{Q}}

\def\bbr{\mathbb{R}}

\def\bbn{\mathbb{N}}

\def\Q{\bbq}
\def\Z{\bbz}
\def\R{\bbr}

\def\N{\bbn}

\def\pcal{\mathcal{P}}

\def\hfrak{\mathfrak{h}}

\def\gfrak{\mathfrak{g}}

\def\vare{\varepsilon}

\def\zg0{Z_{G_\omega}(s)}

\def\zg{Z_G(s)}

\def\be{\begin{equation}}
\def\ee{\end{equation}}

\def\dist{{\rm dist}}

\def\Sob{{\mathcal S}}

\def\dist{d}

\newcommand {\absolute}[1] {\left| {#1} \right|}
\newcommand {\norm}[1] {\left\| {#1} \right\|}

\newcommand{\hide}[1]{}

\newcommand{\sqf}{Q_0}
\newcommand{\cox}{{\mathsf x}}
\newcommand{\coy}{{\mathsf y}}
\newcommand{\coz}{{\mathsf z}}

\newcommand{\torus}{\mathsf M}
\newcommand\ave{\int_{0}^{2\pi}}

\newcommand\rot{r}

\newcommand{\dat}{\Delta(a_t)}
\newcommand{\drot}{\Delta(\rot_\theta)}
\newcommand{\darot}{\Delta(a_t\rot_\theta)}
\newcommand{\Nt}{\mathcal N_t}
\newcommand\ball{\mathsf D}

\begin{document}
\title[Local statistics of the spectrum of a flat torus]{Quantitative equidistribution and the local statistics of the spectrum of a flat torus}

\author{E.~Lindenstrauss}
\address{E.L.: The Einstein Institute of Mathematics, Edmond J.\ Safra Campus, 
Givat Ram, The Hebrew University of Jerusalem, Jerusalem, 91904, Israel}
\email{elon.bl@mail.huji.ac.il}
\thanks{E.L.\ acknowledges support by ERC 2020 grant HomDyn (grant no.~833423).}

\author{A.~Mohammadi}
\address{A.M.: Department of Mathematics, University of California, San Diego, CA 92093}
\email{ammohammadi@ucsd.edu}
\thanks{A.M.\ acknowledges support by the NSF, grants DMS-1764246 and 2055122.}

\author{Z.~Wang}
\address{Z.W.: Pennsylvania State University,
Department of Mathematics, University Park, PA 16802}
\email{zhirenw@psu.edu}
\thanks{Z.W.\ acknowledges support by the NSF, grant  DMS-1753042.}

\begin{abstract}
We show that pair correlation function for the spectrum of a flat $2$-dimensional torus satisfying an 
explicit Diophantine condition agrees with those of a Poisson process with a polynomial error rate.

The proof is based on a quantitative equidistribution theorem and tools from geometry of numbers.  
\end{abstract}

\maketitle

\setcounter{tocdepth}{1}
\tableofcontents

\section{Introduction}

Let $\Delta\subset\R^2$ be a lattice. The eigenvalues of the Laplacian  of the corresponding flat torus $\torus=\R^2/\Delta$ are the values of the quadratic form 
\be\label{eq: def B-Delta}
B_\torus(\cox,\coy)=4\pi^2\norm{\cox v_1+\coy v_2}^2
\ee
at integer points, where $\{v_1,v_2\}$ is a basis for the dual lattice $\Delta^*$.

Let
\[0
=\lambda_0<\lambda_1\leq\lambda_2\cdots
\]
be the corresponding eigenvalues counted with multiplicity. 
By the Weyl's law we have
\[
\#\{j: \lambda_j\leq T\}\sim \tfrac{\vol(\torus)}{4\pi} T.
\]
Let $\alpha<\beta$, and define the {\it pair correlation function}
\[
R_{\torus}(\alpha,\beta,T)=\frac{\#\{(j,k): j\neq k, \lambda_j,\lambda_k\leq T,\; \alpha\leq\lambda_j-\lambda_k\leq \beta\}}{T}.
\]

The following was proved by Eskin, Margulis, and Mozes~\cite{EMM-22}.

\begin{thm}[\cite{EMM-22}, Theorem 1.7]\label{thm: EMM flat torus}
Let $\torus$ be a two dimensional flat torus, and let
\[
B_\torus(\cox,\coy)=\mathsf a\cox^2+2\mathsf b \cox\coy+ \mathsf c\coy^2
\]
be the associated quadratic form giving the Laplacian spectrum of $\torus$, normalized so that $\mathsf a\mathsf c-\mathsf b^2=1$.
Suppose there exist $A\geq 1$ such that for all $(p_1,p_2,q)\in\Z^3$ with $q\geq 2$, we have 
\be\label{eq: EMM Diop cond}
\Bigl|\tfrac{\mathsf b}{\mathsf a}-\tfrac{p_1}{q}\Bigr|+ \Bigl|\tfrac{\mathsf c}{\mathsf a}-\tfrac{p_2}{q}\Bigr|>q^{-A}.
\ee
Then for any interval $[\alpha,\beta]$ with $0\notin[\alpha,\beta]$, we have
\be\label{eq: pair correlation assymp}
\lim_{T\to\infty}R_{\torus}(\alpha,\beta,T)=\pi^2(\beta-\alpha).
\ee
\end{thm}

Prior to~\cite{EMM-22}, Sarnak ~\cite{Sarnak-BT} showed that~\eqref{eq: pair correlation assymp}
holds on a set of full measure in the space of flat tori. The case of inhomogeneous forms, which correspond to eigenvalues of quasi-periodic eigenfunctions, was also studied by Marklof~\cite{Marklof-Inhom, Marklof-Inhom-2}, and by Margulis and the second named author~\cite{MM-Inhom}. More recently, Blomer, Bourgain, Radziwill, and Rudnick \cite{BBRR} studied consecutive spacing 
for certain families of rectangular tori, i.e., $\mathsf b=0$. We also refer
to the work of Str\"ombergsson and Vishe~\cite{Strombergsson-Pankaj} where an effective version of~\cite{Marklof-Inhom} is obtained. 

\medskip

In this paper, we prove a {\em polynomially} 
effective version of Theorem~\ref{thm: EMM flat torus}, i.e., we provide a polynomial error term for $R_{\torus}(\alpha,\beta,T)$.

\begin{thm}\label{thm: main}
Let $\torus$ be a two dimensional flat torus, 
\[
B_\torus(\cox,\coy)=\mathsf a\cox^2+2\mathsf b \cox\coy+ \mathsf c\coy^2
\]
the associated quadratic form giving the Laplacian spectrum of $\torus$ normalized so that $\mathsf a\mathsf c-\mathsf b^2=1$, and let $A\geq10^3$. 
Then there are absolute constants $\delta_0$ and $N$, some $A'$ depending on $A$, and $C$ and $T_0$ depending on $A$, $\mathsf a$, $\mathsf b$, and $\mathsf c$, and for every $0<\delta \leq\delta_0$, a $\kappa=\kappa(\delta,A)$ so that the following holds.

Let $T\geq T_0$, assume that for all $(p_1, p_2, q)\in\Z^3$ with $T^{\delta/A'}<q<T^{\delta}$ we have  
\be\label{eq: Diophantine condition}
\Bigl| \tfrac{\mathsf b}{\mathsf a} -\tfrac{p_1}{q}\Bigr|+ \Bigl|\tfrac{\mathsf c}{\mathsf a}-\tfrac{p_2}{q}\Bigr|> q^{-A}.
\ee
Then if
\[
\absolute{R_{\torus}(\alpha,\beta,T)-\pi^2(\beta-\alpha)} > C (1+\absolute{\alpha}+\absolute{\beta})^{N}T^{-\kappa},
\] 
then there are two primitive vectors $u_1,u_2 \in \Z^2$ so that 
\be\label{eq: exceptional thm main}
\norm{u_1},\norm{u_2} \leq T^{\delta/A} \qquad\text{and}  \qquad\absolute{B_\torus (u_1,u_2)}\leq T^{-1+\delta}
\ee
and moreover 
\[
R_{\torus}(\alpha,\beta,T)-\pi^2(\beta-\alpha)=\frac{M_T(u_1,u_2)}{T} + O\Bigl( (1+\absolute{\alpha}+\absolute{\beta})^{N}T^{-\kappa}\Bigr)
\]
with 
\[
M_T(u_1,u_2) = \#\left\{(\ell_1,\ell_2)\in \tfrac12\Z^2: \begin{array}{l}\ell_1u_1\pm \ell_2u_2\in\Z^2,\\ \ B_\torus (\ell_1u_1\pm \ell_2u_2)\leq T,\\
\ \ {4B_\torus(u_1,u_2)}\ell_1\ell_2 \in {[\alpha,\beta]}\end{array}\right\}.
\]
\end{thm}

\noindent 
The proof of Theorem~\ref{thm: main} is effective, and for all of the above implicit constants, one can give explicit expressions if desired.

Let us now elaborate on the term $M_T(u_1, u_2)$ in the statement of Theorem~\ref{thm: main}: Let $u_1, u_2\in\Z^2$ be two primitive vectors satisfying 
\[
0<\norm{u_i}\leq T^{\delta/A}\quad\text{and}\quad\absolute{B_\torus(u_1,u_2)}\leq T^{-1+\delta}.
\]
Then for all $(\ell_1,\ell_2)\in\tfrac12\Z^2$, we have  
\[
B_\torus(\ell_1u_1+\ell_2u_2)-B_\torus(\ell_1u_1-\ell_2u_2)=
4B_\torus(u_1,u_2)\ell_1\ell_2; 
\]
in particular if $T^{-1-\delta}\leq \absolute{B_\torus(u_1,u_2)}\leq T^{-1+\delta}$ then 
$M_T(u_1,u_2)\gg T^{-10\delta}$, hence the contribution to $R_{\torus}(\alpha,\beta,T)$ from the pairs
\[B_\torus(\ell_1u_1+\ell_2u_2),B_\torus(\ell_1u_1-\ell_2u_2)
\]
would be bigger than any fixed polynomial error term. Moreover, even if~\eqref{eq: Diophantine condition} holds, such pairs $u_1,u_2\in \Z^2$ can definitely exist.
If~\eqref{eq: Diophantine condition} holds, up to changing the order such a pair $u_1,u_2$ is unique --- see Lemma~\ref{lem: super exceptional} --- hence there is no need for additional error terms. 

\medskip

We now state a corollary of Theorem~\ref{thm: main}. 
A rectangular torus has extra multiplicities in the spectrum built in, so to accommodate that we consider the modified pair correlation function 
\[
R'_{\torus}(\alpha,\beta,T)=\frac{\#\{(j,k): \lambda_j\neq\lambda_k<T,\; \alpha\leq\lambda_j-\lambda_k\leq \beta\}}{T}.
\]

\begin{coro}\label{cor: Stronger Diophantine}
Let $\torus$ be a two dimensional flat torus, and let 
\[
B_\torus(\cox,\coy)=\mathsf a\cox^2+2\mathsf b \cox\coy+ \mathsf c\coy^2
\] 
be normalized so that $\mathsf a\mathsf c-\mathsf b^2=1$. 

\begin{enumerate}
\item Suppose there exist $A\geq 1$ and $q>0$ such that for all $(m,n,k)\in\Z^3\setminus\{0\}$, 
\begin{equation}\label{eq: stronger Diophantine}
\absolute{\mathsf a m+\mathsf b n+\mathsf c k}> q\norm{(m,n,k)}^{-A}.
\end{equation}
Then 
\[
\absolute{R_{\torus}(\alpha,\beta,T)-\pi^2(\beta-\alpha)} \leq C (1+\absolute{\alpha}+\absolute{\beta})^{10}T^{-\kappa}.
\] 
 
\item Let $\torus$ be a rectangular torus, i.e., $\mathsf b=0$. Assume there exist $A\geq 1$ and $q>0$ such that for all $(m,n)\in\Z^2\setminus\{0\}$ we have 
\[
\absolute{\mathsf a^2 m+n}> q\norm{(m,n)}^{-A}.
\]
Then 
\[
\absolute{R'_{\torus}(\alpha,\beta,T)-\pi^2(\beta-\alpha)} \leq C (1+\absolute{\alpha}+\absolute{\beta})^{10}T^{-\kappa}.
\]
\end{enumerate}
Where $\kappa$ depends on $A$ and $C$ depends on $\mathsf a$, $\mathsf b$, $\mathsf c$, $A$ and $q$.
\end{coro}

Indeed, under \eqref{eq: stronger Diophantine}, pairs $u_1,u_2$ of primitive integer vectors as in Theorem~\ref{thm: main} do not exist, and if $\torus$ is a rectangular torus the unique (up to order) pair of primitive vectors is given by $e_1=(1,0), e_2=(0,1)$, for which the contribution of $M_T(e_1,e_2)$ can be accounted for by looking at $R'_{\torus}(\alpha,\beta,T)$ instead of $R_{\torus}(\alpha,\beta,T)$.

\medskip
The general strategy of the proof of Theorem~\ref{thm: main} is similar to~\cite{EMM-Upp} and~\cite{EMM-22}. 
That is, we deduce the above theorems from an equidistribution 
theorem for certain unbounded functions in homogeneous spaces.  
Unlike~\cite{EMM-Upp} and~\cite{EMM-22}, where the analysis takes place in the space of unimodular lattices in $\R^4$, 
the homogeneous space in question here is 
\[
X=\SL_2(\R)\times\SL_2(\R)/\Gamma'
\] 
where $\Gamma'$ is a finite index subgroup of $\SL_2(\Z)\times\SL_2(\Z)$. 

This reduction is carried out in~\S\ref{sec: proofs}. The lower bound estimate will be proved using the following effective equidistribution theorem that relies on~\cite[Thm.~1.1]{LMW22}:

Let $G=\SL_2(\R)\times\SL_2(\R)$.
For all $h\in\SL_2(\R)$, we let $\Delta(h)$ denote the element $(h,h)\in G$, and let $H=\Delta(\SL_2(\R))$. For every $t\in\R$ and every $\theta\in[0,2\pi]$, let  
\[
a_t = \begin{pmatrix}e^t &0\\ 0 &e^{-t}\end{pmatrix}\quad\text{and}\quad r_\theta=\begin{pmatrix}\cos\theta &-\sin\theta\\ \sin\theta &\cos\theta\end{pmatrix}.
\]

\begin{thm}\label{thm: r effective equid}
Assume $\Gamma$ is an arithmetic lattice in $G$. 
For every $x_0\in X=G/\Gamma$, and large enough $R$ (depending explicitly on $X$ and the injectivity radius at $x_0$), for any $e^t\geq R^{D}$, at least one of the following holds.
\begin{enumerate}
\item For every $\varphi\in C_c^\infty(X)$ and $2\pi$-periodic smooth function $\xi$ on $\R$, we have 
\[
\biggl|\ave \varphi(\Delta(a_tr_\theta) x_0)\xi(\theta)\diff\!\theta-\int_0^{2\pi}\xi(\theta)\diff\!\theta \int \varphi\diff\!m_X\biggr|\leq \Sob(\varphi)\Sob(\xi)R^{-\kappa_0}
\]
where we use $\Sob(\cdot)$ to denote an appropriate Sobolev norm on both $X$ and $\R$, respectively. 
\item There exists $x\in X$ such that $Hx$ is periodic with $\vol(Hx)\leq R$, and 
\[
\dist_X(x,x_0)\leq R^{D}t^De^{-t}.
\] 
\end{enumerate} 
The constants $D$ and $\kappa_0$ are positive and depend on $X$ but not on $x_0$, and $d_X$ is a fixed metric on $X$. 
\end{thm}

This is a variant of \cite[Thm.\ 1.1]{LMW22}. Indeed, instead of expanding an orbit segment of the unipotent flow $\Delta(u_s)$ where 
\[
u_s=\begin{pmatrix}1 &s\\ 0 &1\end{pmatrix},
\]
here we expand an orbit of the compact group $\{\Delta(r_\theta)\}$. The deduction of Theorem~\ref{thm: r effective equid} from \cite[Thm.\ 1.1]{LMW22} is given in \S\ref{sec: equidistribution} using a fairly simple and standard argument.

\medskip

To prove the upper bound estimate, in addition to Theorem~\ref{thm: r effective equid}, we also need to analyze Margulis functions \'a la ~\cite{EMM-Upp, EMM-22}; our analysis simplifies substantially thanks to simpler structure of the cusp in $\SL_2(\R)\times\SL_2(\R)/\Gamma'$ compared to that in $\SL_4(\R)/\SL_4(\Z)$. This is the content of \S\ref{eq: upper bound}. Indeed Proposition~\ref{prop: main} reduces the analysis to special subspaces, see Definition~\ref{def: special}, that are closely connected to the pairs of almost $B_\torus$-orthogonal vectors discussed above. We study these special subspaces using the elementary Lemma~\ref{lem: Mobius}; in particular, using this lemma we establish Lemma~\ref{lem: special subspace}, which shows that under~\eqref{eq: Diophantine condition} there are at most two special subspaces. Finally, Lemma~\ref{lem: special subspace 2} shows that even for special subspaces, only the range asserted in~\eqref{eq: exceptional thm main} can produce enough solutions to affect the error term.

\subsection*{Acknowledgment}
We would like to thank Jens Marklof for helpful conversations.


\section{Notation and preliminaries}\label{sec: notation}

In this paper 
\[
G=\left\{\begin{pmatrix}g_1 &0\\ 0 &g_2\end{pmatrix}\!:\! g_1,g_2\in \SL_2(\R)\right\}\;\;\text{and}\;\; H=\left\{\begin{pmatrix}g &0\\ 0 &g\end{pmatrix}: g\in \SL_2(\R)\right\}.
\] 
Let $\gfrak=\Lie(G)$ and $\hfrak=\Lie(H)$. 

We identify $G$ with $\SL_2(\R)\times\SL_2(\R)$ and $H$ with 
\[
\{(g,g): g\in \SL_2(\bbr)\}\subset\SL_2(\R)\times\SL_2(\R).
\] 
Indeed, to simplify the notation, we will often denote 
\[
g=\begin{pmatrix}g_1 &0\\ 0 &g_2\end{pmatrix}\in G
\] 
by $(g_1,g_2)$. Given $v=(\cox_1,\coy_1, \cox_2,\coy_2)\in\R^4$, we  write
$g.v=(g_1v_1,g_2v_2)$ where $v_i=(\cox_i,\coy_i)\in\R$ for $i=1,2$ (for purely typographical reasons, we prefer to work with row vectors even though representing these as column vectors would be more consistent).

For all $h\in\SL_2(\R)$, we let $\Delta(h)=(h,h)\in H$. In particular, 
for every $t\in\R$ and every $\theta\in[0,2\pi]$, $\dat$ and $\drot$ denote the images of 
\[
\begin{pmatrix}e^t &0\\ 0 &e^{-t}\end{pmatrix}\quad\text{and}\quad\begin{pmatrix}\cos\theta &-\sin\theta\\ \sin\theta &\cos\theta\end{pmatrix}
\]
in $H$, respectively. 

\subsection{Quadratic Forms}\label{sec: quad form}
Let $\sqf$ denote the {\em determinant} form on $\R^4$:
\[
\sqf(\cox_1,\coy_1, \cox_2,\coy_2)=\cox_1\coy_2-\cox_2\coy_1.
\]
Note that $H= G\cap\SO(\sqf)$.

Let $\Delta\subset\R^2$ be a lattice and let $\Delta^*$ be the dual lattice. 
We normalize $\Delta^*$ to have covolume $(2\pi)^{-2}$ and fix $g_\torus\in\SL_2(\R)$ so that 
\[
2\pi\Delta^*=g_\torus \Z^2. 
\]
The eigenvalues of the Laplacian on $\R^2/\Delta$  are   
$\norm{v}^2$ for $v\in 2\pi\Delta^*$. 
Therefore, given two eigenvalues $\lambda_i=\norm{v_i}^2$, $i=1,2$, we have 
\be\label{eq: B and Q0}
\begin{aligned}
  \lambda_1-\lambda_2&=(\norm{v_1}^2-\norm{v_2}^2)=(v_1+v_2)\cdot(v_1-v_2)\\
  &=\sqf(v_1+v_2, \omega(v_1-v_2))
\end{aligned}
\ee
where $\omega=\begin{pmatrix}0 & -1\\ 1& 0\end{pmatrix}$. 

Recall that $G=\SL_2(\R)\times\SL_2(\R)\subset\SL_4(\R)$. 
Define 
\[
\Lambda = \{(v_1+v_2, \omega(v_1-v_2)): v_1,v_2\in \Z^2\}\subset \R^4.
\]
Then 
$\{(v_1+v_2, \omega(v_1-v_2)): v_1,v_2\in 2\pi\Delta^*\}= (g_\torus, -\omega g_\torus\omega)\Lambda$.

Let $\Gamma'$ be the maximal subgroup of $\SL_2(\Z)\times\SL_2(\Z)$ which preserves $\Lambda$. More explicitly, 
\[
\Gamma'=\{(\gamma_1,\gamma_2)\in\SL_2(\Z)\times\SL_2(\Z): \gamma_1 \equiv \omega \gamma_2 \omega \pmod{2}\}.
\]
Let $X=G/\Gamma'$.

\subsection*{M\"{o}bius transformations}
In this section, we prove an elementary lemma concerning M\"{o}bius transformations. This lemma will be used to complete the proof of Lemma~\ref{lem: super exceptional}; it also will be used in the proof of Lemma~\ref{lem: not many solutions}.  

Let $\pcal$ denote the set of primitive vectors in $\Z^2$. For every $t\geq 1$, let 
\[
\pcal(t)=\{v\in\pcal: \norm{v}< e^t\}.
\]

\begin{lemma}\label{lem: Mobius}
Let $A\geq 10^3$, $s>0$ and $0<\eta^A<e^{-s/100}$. 
Assume that for $i=1,2$ there are $v_i, v_i', v_i''\in \mathcal P(s)$ satisfying
\be\label{eq: Mobius}
1\leq\absolute{\sqf(v,w)}\ll \eta^{-4}, \qquad \text{for $v, w\in\{v_i,v_i',v_i''\}$.}
\ee
Also suppose there are $h\in\PGL_2(\R)$ and $C>0$ so that 
\be\label{eq: Mobius h}
hv_1=\mu v_2 + w_{1,2},\;\; hv_1'=\mu' v_2' + w'_{1,2},\;\; hv_1''=\mu'' v_2'' + w_{1,2} 
\ee
where $\absolute{\mu},\absolute{\mu'},\absolute{\mu''}\geq C^{-1}$ and $\norm{w}\leq C\eta^Ae^{-s}$ for $w\in\{w_{1,2}, w'_{1,2}, w''_{1,2}\}$. 

Then 
there exists $Q\in\Mat_2(\Z)$ with $\norm{Q}\leq\eta^{-100}$ and $\lambda\in\R$ 
so that 
\[
\norm{h-\lambda Q}\leq C'\eta^{A-50},
\]
where $C'$ depends on $C$ and polynomially on $\norm{h}$. 
\end{lemma}

\begin{proof}
Let us write $v_i=(\cox_i,\coy_i)$, $v_i'=(\cox_i',\coy_i')$, and $v_i''=(\cox_i'',\coy_i'')$. The matrix 
\[
M_1=\begin{pmatrix}\coy_1 & -\cox_1\\ \coy'_1\mathsf z_1 & -\cox'_1\mathsf z_1\end{pmatrix} \qquad \text{for } \mathsf z_1=\tfrac{\cox_1''\coy_1-\cox_1\coy_1''}{\cox_1''\coy_1'-\cox_1'\coy_1''}
\]
acting on $\mathbb P^1$ takes $(\cox_1:\coy_1)$ to $(0:1)$, $(\cox'_1:\coy'_1)$ to $(1:0)$ and $(\cox''_1:\coy''_1)$ to $(1:1)$. The matrix
\[
M_2=\begin{pmatrix}-\cox'_2\mathsf z_2 & \cox_2\\ -\coy'_2\mathsf z_2 & \coy_2\end{pmatrix} \qquad \text{for }\mathsf z_2=\tfrac{\cox_2''\coy_2-\cox_2\coy_2''}{\cox_2''\coy_2'-\cox_2'\coy_2''}
\]
in turn takes $(0:1)$ to $(\cox_2:\coy_2)$, $(1:0)$ to $(\cox'_2:\coy'_2)$ and $(1:1)$ to $(\cox''_2:\coy''_2)$.
By~\eqref{eq: Mobius}, we have that $r=\absolute{\det(M_1)\det(M_2)}^{-1}$ is a rational number of height $\ll \eta^{-20}$. Thus by~\eqref{eq: Mobius h} 
\begin{align*}
&h=\pm\sqrt{r}M_2M_1 + O(\eta^{A-50}) \quad\text{or}\\
&h=\pm\sqrt{r}\begin{pmatrix} 1 &0\\ 0 &-1\end{pmatrix} M_2M_1 + O(\eta^{A-50}).
\end{align*}
Since the denominators of the entries of $M_1$ and $M_2$ are bounded by $\eta^{-4}$, and since all our implicit constants are allowed to depend on $\norm{h}$, we may conclude the claim.
\end{proof}

We draw some corollaries of Lemma~\ref{lem: Mobius}. 

\begin{definition}\label{def: special}
Let $g=(g_1,g_2)\in G$.
A two dimensional $g\Z^4$-rational linear subspace $L\subset\R ^4$ is called $(\rho, A,t)$-exceptional if 
there are $(v_1,0), (0,v_2)\in\Z^4$ satisfying  
\be\label{eq: quasi-null 1}
\norm{g_1v_1}, \norm{g_2v_2}\leq e^{\rho t} \quad\text{and}\quad \absolute{\sqf(g_1v_1, g_2v_2)}\leq e^{-A\rho t}
\ee
so that $L \cap g\Z^4$ is spanned by $\{(g_1v_1,0), (0,g_2v_2)\}$.

Given a $(\rho, A,t)$-special subspace $L$, we will refer to $\{(g_1v_1,0), (0,g_2v_2)\}$ as a {\em spanning set} for $L$.   
\end{definition}

\begin{lemma}\label{lem: special subspace}
Let $A\geq 10^{3}$, and let $g=(g_1,g_2)\in G$. Let $\rho\leq A/100$. Then for all $t$ large enough, depending on $\norm{g}$, at least one of the following holds:
\begin{enumerate}
\item There are at most two different $(\rho, A, t)$-exceptional subspaces.
\item There exists $Q\in\Mat_2(\Z)$ whose entries are bounded by $e^{100\rho t}$ and $\lambda\in\R$ 
satisfying $\norm{g_2^{-1}g_1-\lambda Q}\leq e^{-(A-100)\rho t}$.
\end{enumerate} 
\end{lemma}  

\begin{proof}
We begin by proving the first assertion in the lemma. 
Let $\eta=e^{-\rho t}$ and $s=\rho t$. 
Indeed assume there are three different $(\rho, A, t)$-special subspaces in $\R^4$, and let $v_i, v_i', v_i''\in\pcal_{s}$, $i=1,2$, be the corresponding spanning vectors. 
Then 
\[
1< |\sqf(v,w)|\ll e^{2\rho t},\qquad \text{for $v,w\in\{v_1,v'_1,v_1''\}$}
\]
That is, $\{v_1,v_1',v_1''\}$ satisfies \eqref{eq: Mobius} with $\eta=e^{-\rho t}$ so long as $t$ is large enough to account for the implied constant.  
Moreover, if we put $h=g_2^{-1}g_1$, then   
\[
hv_1=\mu v_2 + w_{1,2}
\]
where $\mu\in\R$ satisfies $\absolute{\mu}\gg 1$ and 
$\norm{w_{1,2}}\ll e^{-A\rho t}=\eta^{(A-1)}e^{-s}$ (recall that the implicit constants in these inequalities are allowed to depend polynomially on $\norm{g_1}$ and $\norm{g_2}$). Similarly, 
\[
hv'_1=\mu' v'_2 + w'_{1,2}\quad\text{and}\quad hv''_1=\mu'' v''_2 + w''_{1,2}
\]  
where $\mu',\mu''\in\R$ satisfy $\absolute{\mu'},\absolute{\mu''}\gg 1$ and 
$\norm{w'_{1,2}}, \norm{w''_{1,2}}\ll e^{-A\rho t}$. 
Therefore, $\{v_2,v_2',v_2''\}$ also satisfy \eqref{eq: Mobius}. Moreover, 
$h=g_2^{-1} g_1$ satisfies~\eqref{eq: Mobius h} with $A-1$, $\eta$, and $s$, in view of the above discussion. Hence, Lemma~\ref{lem: Mobius} implies that the assertion in part~(2) of this lemma holds so long as $t$ is large enough.   
\end{proof}

\subsection*{Special subspaces and the spectrum of flat tori}
Using the discussion in \S\ref{sec: quad form}, we translate the conclusion of Lemma~\ref{lem: special subspace} to a similar statement about the quadratic form $B_\torus$.

\begin{lemma}\label{lem: super exceptional}
Let $A\geq 10^{4}$, and let $\rho\leq A/100$. 
Recall that 
\[
B_\torus(\cox, \coy)=\mathsf a\cox^2+2\mathsf b \cox\coy+\mathsf c\coy^2
\]
is renormalized so that $\mathsf a\mathsf c-\mathsf b^2=1$. 
Then for all $t\geq t_0$, depending on $\rho$, $\absolute{\mathsf a}$, $\absolute{\mathsf b}$, and $\absolute{\mathsf c}$, at least one of the following holds:
\begin{enumerate}
\item There is a unique, up to change of order, pair of primitive vectors $u_1, u_2\in\Z^2\setminus\{0\}$ satisfying 
\[
\norm{u_i}\leq e^{\rho t}\quad\text{and}\quad |B_\torus( u_1,u_2)|\leq e^{-A\rho t}
\]
    \item There exists $Q\in\Mat_2(\Z)$ whose entries are bounded by $e^{100\rho t}$ and $\lambda\in\R$ 
satisfying 
\[
\norm{\begin{pmatrix}\mathsf a & \mathsf b\\ \mathsf b &\mathsf c\end{pmatrix}-{\lambda}Q}\leq e^{-(A-100)\rho t}.
\]
\end{enumerate}

In particular, if $\torus$ is a rectangular torus, i.e., $\mathsf b=0$, then $t_0$ may be chosen so that
if part~(2) is not satisfied, then (up to changing the order) $u_1=(1,0)$ and $u_2=(0,1)$.  
\end{lemma}

\begin{proof}
Let $t_1$ be large enough so that Lemma~\ref{lem: special subspace} holds for all $t\geq t_1$. 
Since $B_\torus$ is positive definite, there exists $t'_0$ so that if $t\geq t'_0$, then 
\[
\absolute{B_\torus(u_1,u_2)}<e^{-A\rho t}
\]
implies that $\{u_1,u_2\}$ is linearly independent. 

Let $t_0=\max(t_1,t'_0)$ and let $t\geq t_0$. Put
\[
g=(g_1,1)=\left(\begin{pmatrix}\mathsf a & \mathsf b\\ \mathsf b & \mathsf c\end{pmatrix}, 1\right). 
\]
Note that if part~(2) in Lemma~\ref{lem: special subspace} holds, then part~(2) in this lemma holds and the proof is complete. Thus let us assume that part~(1) in Lemma~\ref{lem: special subspace} holds. 

Let $u_i=(\cox_i,\coy_i)\in\Z^2\setminus\{0\}$. Then  
\begin{align*}
B_\torus(u_1,u_2)&=(\cox_1,\coy_1)\begin{pmatrix}\mathsf a & \mathsf b\\ \mathsf b & \mathsf c\end{pmatrix}
\begin{pmatrix}\cox_2 \\ \coy_2\end{pmatrix}\\
&=(\mathsf a\cox_1+\mathsf b\coy_1)\cox_2+(\mathsf b\cox_1+\mathsf c\coy_1)\coy_2\\
&=\left(\begin{pmatrix}\mathsf a & \mathsf b\\ \mathsf b & \mathsf c\end{pmatrix}\begin{pmatrix}\cox_1 \\ \coy_1\end{pmatrix}\right)\wedge \begin{pmatrix}-\coy_2 \\ \cox_2\end{pmatrix}=\sqf\bigl(g_1(\cox_1,\coy_1), (-\coy_2,\cox_2)\bigr).
\end{align*}
Thus if $u_1, u_2$ satisfy part~(1), then $(g_1(\cox_1,\coy_1), (0, 0))$ and $((0, 0),(-\coy_2,\cox_2))$ span a $(\rho, A, t)$-special subspace for $g\Z^4$.

By Lemma~\ref{lem: special subspace}, there is at most two such subspaces. 
Moreover, since $B_\torus(\;,\;)$ is symmetric, we conclude that 
\[
\sqf(g_1(\cox_2,\coy_2),(-\coy_1,\cox_1))=\sqf\bigl(g_1(\cox_1,\coy_1), (-\coy_2,\cox_2)\bigr).
\]
This implies the two special subspaces are spanned by 
\begin{align*}
&\{(g_1(\cox_1,\coy_1), 0, 0), (0,0,-\coy_2,\cox_2)\}\quad\text{or}
&\{(g_1(\cox_2,\coy_2), 0, 0), (0,0,-\coy_1,\cox_1)\}.
\end{align*}
This shows part~(1) in this lemma holds.

Assume now that $\mathsf b=0$, and suppose part~(2) does not hold. Let $u_i$ be as in part~(1). Then 
\be\label{eq: approximate c}
\absolute{B_\torus(u_1,u_2)}=\absolute{\mathsf a\cox_1\cox_2+\mathsf a^{-1}\coy_1\coy_2}\leq e^{-A\rho t}. 
\ee
Unless $\coy_1\coy_2=0$, the above contradicts that part~(2) does not hold.
Therefore, we have $\coy_1\coy_2=0$. Assuming $t$ is large enough so that the right side in~\eqref{eq: approximate c} is $<\absolute{\mathsf a}$, we conclude $\cox_1\cox_2=0$ and the claim follows.   
\end{proof}

The following lemma further investigates the contribution of special subspaces, or more precisely, vectors $u_1, u_2$ satisfying part~(1) in Lemma~\ref{lem: super exceptional}. We note that condition~\eqref{eq: special but not exceptional} is~\eqref{eq: exceptional thm main} in Theorem~\ref{thm: main}.

\begin{lemma}\label{lem: special subspace 2}
Let $A\geq 10^3$ and $0<\rho<1/(100 A)$.
Let 
\[
B_\torus(\cox, \coy)=\mathsf a\cox^2+2\mathsf b \cox\coy+\mathsf c\coy^2
\]
which is normalized so that $\mathsf a\mathsf c-\mathsf b^2=1$.
The following holds for all large enough $t$, depending on $\rho$, $\absolute{\mathsf a}$, $\absolute{\mathsf b}$, and $\absolute{\mathsf c}$. 
Let $u_1, u_2\in\Z^2\setminus\{0\}$ satisfy 
\[\norm{u_i}\leq e^{\rho t} \qquad\text{and}\qquad \absolute{B_\torus(u_1,u_2)}\leq e^{-A\rho t}.\]
Assume further that   
\be\label{eq: special but not exceptional}
\absolute{B_\torus(u_1, u_2)}> e^{(-2+2\rho)t}. 
\ee
Let $C>0$, then 
\[
\#\left\{(\ell_1,\ell_2)\in \tfrac12 \Z^2: \begin{array}{l}\absolute{\ell_i}\leq Ce^{t}\qquad\\ \ 4{B_\torus(u_1,u_2)}\ell_1\ell_2 \in {[\alpha,\beta]}\end{array}\right\}\ll \max(\absolute{\alpha},\absolute{\beta}) e^{(2-\rho) t}
\] 
where the implied constant depends on $C$, ${\mathsf a}$, ${\mathsf b}$, and ${\mathsf c}$. 
\end{lemma}

\begin{proof} 
Let $(\ell_1, \ell_2)$ satisfy that
$|\ell_i|\leq C e^t$ and 
\be\label{eq: compute sqrf v}
4B_\torus(u_1,u_2)\ell_1\ell_2 \in [\alpha,\beta].
\ee
Then the number of solutions with $\ell_1=0$ or $\ell_2=0$ is $\ll e^t$. Therefore, we assume $\ell_i\neq0$ for $i=1,2$ for the rest of the argument. 

Assume that  
\[
\absolute{B_\torus(u_1,u_2)}> e^{(-2+2\rho)t}.
\]
Then~\eqref{eq: compute sqrf v} implies that 
\be\label{eq: prod r1r2}
0 < 4\absolute{\ell_1\ell_2}\leq \max(\absolute{\alpha}, \absolute{\beta}) e^{(2-2\rho)t}.
\ee
The number of $(\ell_1,\ell_2)\in\Z^2$ with $0<\absolute{\ell_1}\leq Ce^t$ so that~\eqref{eq: prod r1r2} holds is 
\[
\ll  \max(\absolute{\alpha}, \absolute{\beta}) t e^{(2-2\rho)t}\ll \max(\absolute{\alpha}, \absolute{\beta}) e^{(2-\rho)t}
\]
as we claimed. 
\end{proof}


\section{Circular averages and values of quadratic forms}\label{sec: proofs}
In this section, we state an equidistribution result for the action of $\SO(\sqf)$. 
Theorem~\ref{thm: main} will be deduced from this equidistribution theorem in \S\ref{sec: proof of main} using some preparatory lemmas which will be established in this section.   

Let $f_i$ be compactly supported bounded Borel functions on $\R^2$, and define $f$ on $\R^4$ by $f(w_1,w_2)=f_1(w_1)f_2(w_2)$.  
For any $g'\in G$, let
\be\label{eq: def hat f}
\hat f(g'\Gamma')=\sum_{v\in g'\Lambda_{\rm nz}} f(v)
\ee
where 
\begin{align*}
&\Lambda = \{(v_1+v_2, \omega(v_1-v_2)): v_1,v_2\in \Z^2\}\subset \R^4,\\  
&\Lambda_{\rm nz}=\{(w_1,w_2)\in\Lambda: w_1\neq 0\text{ and } w_2\neq 0\}\\
&\Gamma'= \{(\gamma_1,\gamma_2)\in\SL_2(\Z)\times\SL_2(\Z): \gamma_1 \equiv \omega \gamma_2 \omega \pmod{2}\},
\end{align*}
and $\omega=\begin{pmatrix}0 & -1\\ 1& 0\end{pmatrix}$. Note that $\Gamma'$ preserves $\Lambda$ and $\Lambda_{\rm nz}$. 

Let $X=G/\Gamma'$, and let $m_X$ denote the $G$-invariant probability measure on $X$.

\begin{thm}\label{thm: equi of f hat}
For every $A\geq 10^4$ and $0<\rho\leq 10^{-4}$, there exist $\hat A$ (depending on $A$) and $\delta_1, \delta_2$ (depending on $\rho$ and $A$) with 
\[
\rho/\hat A\leq \delta_1/A\leq \rho/100,
\]
so that for all $g=(g_1,g_2)\in G$ and all large enough $t$, depending linearly on $\log(\norm{g_i})$, the following holds. 

Assume that for every $Q\in\Mat_2(\Z)$ with $e^{\rho t/\hat A}\leq \norm{Q}\leq e^{\rho t}$ and all $\lambda\in\R$, we have  
\be\label{eq: Diophantine main them Q0}
\norm{g_2^{-1}g_1-{\lambda}Q}>\norm{Q}^{-A/1000}.
\ee
There exists some $C'$ depending on $A$ and polynomially on $\norm{g_i}$ so that the following holds. For any $2\pi$-periodic smooth function $\xi$ on $\R$, if
\begin{equation*}
\biggl|\int_0^{2\pi}\hat f(\Delta(a_t\rot_\theta) g\Gamma')\xi(\theta)\diff\!\theta-\int_0^{2\pi}\xi\diff\!\theta\int_X\hat f\diff\!m_X\biggr|
> C'\Sob(f)\Sob(\xi)e^{-\delta_2 t}
\end{equation*}
then there are at least one, and at most two, $(\delta_1/A, A, t)$-exceptional subspaces, say $L$ and  $L'$ (for notational convenience, if there is only one exceptional subspace, set $L'=L$). Moreover  
\begin{multline*}
\int_0^{2\pi}\hat f(\Delta(a_t\rot_\theta) g\Gamma')\xi(\theta)\diff\!\theta=\int_0^{2\pi}\xi\diff\!\theta\int_X\hat f\diff\!m_X \\
+ \mathcal M + O(\Sob(f)\Sob(\xi)e^{-\delta_2 t})
\end{multline*}
where 
\[
\mathcal M=\int_{\mathcal C}\hat f_{\rm sp}(\theta)\xi(\theta)\diff\!\theta
\]
with 
\begin{align*}
\hat f_{\rm sp}(\theta)&=\!\!\!\sum_{v\in g\Lambda_{\rm nz} \cap (L\cup L')} f(\Delta(a_t r_\theta)v)\\
\mathcal C&=\left\{\theta\in[0,2\pi]: \hat f_{\rm sp}(\theta)\geq e^{\delta_1 t}\right\}
\end{align*}
\end{thm}

The proof of Theorem~\ref{thm: equi of f hat} will be completed in \S\ref{sec: proof of equi}; 
it relies on results in \S\ref{sec: equidistribution} and \S\ref{eq: upper bound}. 
The goal in the remaining parts of this section and \S\ref{sec: proof of main} is to complete the proof of Theorem~\ref{thm: main} 
using Theorem~\ref{thm: equi of f hat}. We will also explicate the proof of Corollary~\ref{cor: Stronger Diophantine} at the end of \S\ref{sec: proof of main}.  

Before proceeding, however, let us record an a priori, i.e., without assuming~\eqref{eq: Diophantine main them Q0}, 
upper bound for $\int_0^{2\pi}\hat f(\Delta(a_t\rot_\theta) g\Gamma')\diff\!\theta$.  

\begin{lemma}\label{lem: gen bd for int hat f}
For every $0<\eta<1$, there exists $t_{\eta}\ll\absolute{\log\eta}$ so that 
the following hold. 
Let $g=(g_1, g_2)\in G$ and $R\geq 1$; assume that $\norm{g_i}\leq R$.
Let $f_i$ be the characteristic function of $\{w\in\R^2: \norm{w}\leq R\}$, and put $f=f_1f_2$. 
\begin{enumerate}
\item For every $t\geq t_\eta$ we have  
\[
\int_0^{2\pi}\hat f(\Delta(a_t\rot_\theta) g\Gamma')\diff\!\theta \ll e^{\eta t}
\]  
\item Let $t\geq t_\eta$. Let $L\subset\R^4$ be a two dimensional subspace so that 
$L\cap g\Z^4$ is spanned by $\{(g_1v_1,0), (0,g_2v_2)\}$ for $(v_1,0), (0,v_2)\in\Z^4\setminus\{0\}$. Then 
\[
\int_{[0,2\pi]\setminus\mathcal C_L}\hat f_L(\theta)\diff\!\theta \ll e^{(-1+\eta) t}
\]
where $\hat f_L(\theta)=\sum_{v\in g\Lambda_{\rm nz}\cap L} f(\Delta(a_t r_\theta)v)$ and 
\[
\mathcal C_L=\{\theta\in[0,2\pi]: \hat f_L(\theta)\geq e^{\eta t}\}.
\] 
\end{enumerate}
The implied constants depend polynomially on $R$.
\end{lemma}

We postpone the proof of this lemma to the end of \S\ref{eq: upper bound}. 
Part~(1) in this Lemma should be compared with~\cite[Lemma 5.13]{EMM-Upp}; 
indeed in loc.\ cit.\ the integral appearing part~(1) in Lemma~\ref{lem: gen bd for int hat f} is bounded by $O(t)$ (vs.\ $e^{o(t)}$ that we give here) which is sharp. The above however suffices for our needs.

\subsection{A linear algebra lemma}\label{sec: lower bd}
The goal in the remaining parts of this section is to relate the circular integrals as appear in Theorem~\ref{thm: equi of f hat} to the counting problem in Theorem~\ref{thm: main}. This is the content of Lemma~\ref{lem: rotating to vertical} 
which should be compared with~\cite[Lemma 3.6]{EMM-Upp} and~\cite[Lemma 3.4]{EskinMasur-Upp}. 
We will also establish a certain upper bound estimate in Lemma~\ref{lem: gen upp bd} 
which will be used in the proof of Theorem~\ref{thm: main}.

Let us begin by fixing some notation which will be used in Lemma~\ref{lem: rotating to vertical 2} and Lemma~\ref{lem: rotating to vertical}. 
Let $\alpha<\beta$, $R\geq \max\{1,\absolute{\alpha},\absolute{\beta}\}$, $R^{-1}\leq q\leq R$, and $0<\vare<R^{-4}$.
Let $\varrho:\R \to [0,1]$ be a smooth function supported on $[q-\vare,q]$. 
Let $f_{1}$ be a smooth function on $\R^2$ satisfying  
\be\label{eq: def f1}
1_{[-\frac\vare2, \frac\vare2]}(\cox) \cdot \varrho(\coy)\leq f_1(\cox,\coy)\leq 1_{[\frac{-\vare-\vare^2}2, \frac{\vare+\vare^2}2]}(\cox) \cdot \varrho(\coy);
\ee
we chose $\varrho$ and $f_1$ so that their partial derivatives are $\ll_R \vare^{-10}$.

For an interval $I=[a,b]$ and $\delta>0$, put 
\begin{equation}\label{eq:modified intervals}
    \begin{aligned}
I_{\delta}&=[a-\delta,b+\delta] \supset I \\
I_{-\delta}&=[a+\delta,b-\delta] \subset I.
\end{aligned}
\end{equation}
Given two intervals $I\subset [-R^2,R^2]$ and $I'\subset[0,R]$, 
let $f_{I,I'}$ be a smooth function with partial derivatives $\ll_R\vare^{-10}$ satisfying
\be\label{eq: def f2}
1_{I^{(1)}}(\cox) \cdot 1_{{I'}^{(1)}}(|\coy|)\leq f_{I,I'}(\cox,\coy)\leq 1_{I^{(2)}}(\cox)\cdot  1_{{I'}^{(2)}}(|\coy|),
\ee
where we write $I^{(k)}=I_{10 kR^3\vare}$ (in the formula above we used $k=1,2$, but later also larges values of $k$ will be used).

\medskip

For any function $h$ on $\R^2$, define 
\[
J_{h}(\coy)=\int_\R h(\cox,\coy)\diff\!\cox.
\]
Note that if $f_1$ is as in~\eqref{eq: def f1}, then 
\be\label{eq: Jf1}
J_{f_1}(\coy)=\varrho(\coy)(\vare+O(\vare^2))
\ee

Let $f_1$ be as above (for this $q$ and some $\varrho$) and let 
$f_2=f_{I_0,I_1}$ (for $I_0=[-q^{-1}\beta,-q^{-1}\alpha]$ and some $I_1 \subset [0,R]$).  
Define $f$ on $\R^4$ by 
\[
f(v_1,v_2)=f_{1}(v_1)f_{2}(v_2).
\]  
We will work with a slight variant of polar coordinates in $\R^2$:
$0\neq w\in\R^2$ is denoted by $(\theta_w, \norm{w})$
where $\theta_w\in[0,2\pi]$ is so that $r_{\theta_w}w=(0,\norm{w})$.

\begin{lemma}
\label{lem: rotating to vertical}
Let the notation be as above. 
Let $t>\log(4R^3\vare^{-2})$, and let $\xi$ be a $2\pi$-periodic non-negative smooth function.
Let $v=(v_1, v_2)\in\mathbb R^4$ with $\norm{v_i}\geq R^{-1}$. Then  
\begin{multline}\label{eq:integral algebraic lemma'} 
qe^{2t}\int_0^{2\pi} f(\darot v)\xi(\theta)\diff\!\theta \leq \\
\begin{cases} (1+O(\vare))J_{f_1}\Bigl(e^{-t}\norm{v_1}\Bigr)\xi(\theta_1) + O\Bigl(\Lip(f_1)\Lip(\xi)e^{-2t}\Bigr) & \text{if~\eqref{eq: mathcal E} holds}
\\ 0&\text{otherwise}
\end{cases}
\end{multline}
where 
\be\label{eq: mathcal E}
\Bigl(-q^{-1}\sqf(v), e^{-t}\norm{v_2}\Bigr)\in I_0^{(3)}\times I_1^{(3)}\quad \text{and}\quad \norm{v_1}\leq 2Re^t.
\ee
If we moreover assume that $e^{-t}\norm{v_2}\in I_1$ and $\sqf(v)\in[\alpha,\beta]$, then 
\begin{multline}\label{eq:integral algebraic lemma}
qe^{2t}\int_0^{2\pi} f(\darot v)\xi(\theta)\diff\!\theta=\\
(1+O(\vare))J_{f_1}\Bigl(e^{-t}\norm{v_1}\Bigr)\xi(\theta_1) f_2\Bigl(-q^{-1}\sqf(v), e^{-t}\norm{v_2}\Bigr)\\ + O\Bigl(\Lip(f_1)\Lip(\xi)e^{-2t}\Bigr).
\end{multline} 
The implied constants depend polynomially on $R$. 

Analogous statements hold with the roles of $v_1$ and $v_2$ switched.  
\end{lemma}

The proof is based on a direct computation which we will carry out in the next lemma.

\begin{lemma}
\label{lem: rotating to vertical 2}
Let the notation be as in Lemma~\ref{lem: rotating to vertical}.
Let $t>\log(4R^3\vare^{-2})$. If
\[
f(\darot v)\neq 0
\]
for some $v=(v_1, v_2)\in\R^4$ with $\norm{v_i}\geq R^{-1}$ and some $\theta\in[0,2\pi]$, then all of the following properties hold 
\begin{enumerate}
\item $q(1-2\vare)\leq e^{-t}\norm{v_1}\leq q(1+\vare)$.
\item $\absolute{\theta-\theta_{v_1}}\leq 2R\vare e^{-2t}$,  
\item $e^{-t}\norm{v_2}\in I_{1}^{(2)}$, and 
\item $-q^{-1}\sqf(v)\in I_0^{(3)}$.
\end{enumerate} 
\end{lemma}

\begin{proof}
The definitions of $f_1$ and $f_2$ imply that 
\[
\text{ if $\quad\norm{v_i}> (R+20R^3\vare)e^t,\quad$ then $\quad f(\darot v)=0$}
\]
and there is nothing to prove. We thus assume that $\norm{v_i}\leq (R+20R^3\vare)e^t$ for the rest of the argument. 

For convenience, we will write $\theta_1=\theta_{v_1}$. Since $\theta\in[0,2\pi]$ satisfies 
\[
a_tr_\theta v_1\in [\tfrac{-\vare-\vare^2}2, \tfrac{\vare+\vare^2}2]\times [q-\vare,q]
\] 
only if 
\be \label{eq: range of theta 2}
\absolute{\theta-\theta_1}\leq \tfrac32\vare e^{-t}\norm{v_1}^{-1}\leq 2R\vare e^{-2t},
\ee
we see that when 
\be\label{eq: norm of v1 and et 2}
q(1-2\vare)\leq e^{-t}\norm{v_1}\leq q(1+\vare)
\ee 
fails, $f(\darot v)=0$.

Thus, assume that~\eqref{eq: range of theta 2} and~\eqref{eq: norm of v1 and et 2} hold for the rest of the argument, which is to say the conditions (1) and (2) in the lemma are satisfied if $f(\darot v)\neq0$. We now show (3) and (4) must also hold.    

Let us write 
\[
\rot_{\theta_1}v_2=(\bar\cox_2,\bar\coy_2).
\]
Recall that $\norm{v_i}\leq (R+20R^3\vare)e^t$ and that $\theta$ is in the range \eqref{eq: range of theta 2}, and write 
\[
r_\theta v_1=(\cox_1', \coy_1')\quad\text{and}\quad r_\theta v_2=(\cox_2', \coy_2').
\]
Then $\absolute{\cox'_1}\leq 4R \vare e^{-t}$, $\absolute{\coy_1'-\norm{v_1}}\leq 4R \vare e^{-t}$, 
\be\label{eq: x2 bar and x2' 2}
\absolute{\cox'_2-\bar\cox_2},\absolute{\coy_2'-\bar\coy_2}\leq 3R\vare e^{-t}\norm{v_1}^{-1}\norm{v_2}\leq 4R^3\vare e^{-t};
\ee 
in the last inequality we used $\vare<R^{-4}$, $\norm{v_2}\leq (R+20R^3\vare)e^t$, and~\eqref{eq: norm of v1 and et 2}.  

Thus, we conclude that 
\[
a_tr_\theta v_2=(e^t\cox'_2, e^{-t}\coy'_2)=(e^t\bar\cox_2+\cox_{2,\theta}, e^{-t}\bar\coy_2+\coy_{2,\theta})
\]
where $|\cox_{2,\theta}|\leq 4R^3\vare$ and $|\coy_{2,\theta}|\leq 4R^3\vare e^{-2t}$. 

In view of the definition of $f_2$, we conclude that $f_2(a_tr_\theta v_2)=0$, unless 
\[
e^t\cox_2'\in (I_0^{(1)})_{20R^3\vare}\quad\text{and}\quad  e^{-t}\coy_2'\in (I_1^{(1)}){20R^3\vare +\vare}
\]
These and the bound on $\cox_{2,\theta}$ imply that 
\be\label{eq: norm of cox2 2}
e^t\bar\cox_2\in (I_0^{(1)})_{24R^3\vare}
\ee
and hence using the upper bound on $|\bar\cox_2|$ implied by~\eqref{eq: norm of cox2 2}, we get
\be\label{eq: norm of coy2 2}
\absolute{\absolute{\bar\coy_2}-\norm{v_2}}\leq\tfrac{R^4e^{-2t}}{\norm{v_2}}.
\ee
Since $e^{-t}\coy_2'\in (I_1^{(1)})_{20R^3\vare +\vare}$ and $|\coy_{2,\theta}|\leq 4R^3\vare e^{-2t}$, 
we conclude from~\eqref{eq: norm of coy2 2} that if $f_2(a_tr_\theta v_2)\neq 0$, then 
\[
e^{-t}\norm{v_2}\in (I_1^{(1)})_{21R^3\vare}
\]
which establishes (3) in the lemma.

Finally, combining~\eqref{eq: norm of cox2 2} and~\eqref{eq: norm of v1 and et 2}, we conclude that 
\[
q^{-1}\norm{v_1} \bar\cox_2\in (I_0^{(1)})_{30R^3\vare}.
\]
Since $\drot\in\SO(Q_0)$ for all $\theta$ and $\Delta(r_{\theta_1})v= (0, \norm{v_1}, \bar\cox_2,\bar\coy_2)$, we get   
\[
-q^{-1}\sqf(v)=-q^{-1}\sqf(\Delta(r_{\theta_1}v))=q^{-1}\norm{v_1}\bar\cox_2\in I_0^{(2)}.
\]
as it was claimed in~(4). 
\end{proof}

We now turn to the proof of Lemma~\ref{lem: rotating to vertical}

\begin{proof}[Proof of Lemma~\ref{lem: rotating to vertical}]
For convenience we write $\theta_1=\theta_{v_1}$. By Lemma~\ref{lem: rotating to vertical 2} if $f(\darot v)\neq0$, then all the following hold true: 
\begin{subequations}
\begin{align}
\label{eq: norm of v1 and et} &q(1-2\vare)\leq e^{-t}\norm{v_1}\leq q(1+\vare)\\
\label{eq: range of theta}&\absolute{\theta-\theta_1}\leq 2R\vare e^{-2t}\\
\label{eq: norm of v2} &e^{-t}\norm{v_2}\in I_{1}^{(3)}\\
\label{eq: value of sqf(v)}& -q^{-1}\sqf(v)\in I_0^{(3)}
\end{align} 
\end{subequations}

We begin with the following computation which will be used in the proof of both~\eqref{eq:integral algebraic lemma'} and~\eqref{eq:integral algebraic lemma}.
\[
\int_0^{2\pi} f_1(a_tr_\theta v_1)\diff\!\theta=\int_0^{2\pi}  f_1(-e^t\norm{v_1}\sin\theta, e^{-t}\norm{v_1}\cos\theta)\diff\!\theta.
\]
Making the change of variable $\coz=-e^t\norm{v_1}\sin\theta$, the above integral equals
\begin{multline}\label{eq: def J used}
 \frac{e^{-t}}{\norm{v_1}}\int_{-\infty}^\infty f_1\biggl(\coz, e^{-t}\norm{v_1}\sqrt{1-(e^{-t}\coz/\norm{v_1})^2}\biggr)\frac{1}{\sqrt{1-(e^{-t}\coz/\norm{v_1})^2}}\diff\!\coz\\
 =\frac{e^{-t}}{\norm{v_1}}\int_{-\infty}^\infty f_1(\coz,e^{-t}\norm{v_1})\diff\!\coz+O(R^2\Lip(f_1)e^{-4t})\\
 =q^{-1}(1+O(\vare))e^{-2t}J_{f_1}(e^{-t}\norm{v_1})+ O(R^2\Lip(f_1)e^{-4t})
\end{multline}
where in the last equality we used~\eqref{eq: norm of v1 and et} and~\eqref{eq: Jf1}.

Let us now begin the proof of \eqref{eq:integral algebraic lemma'}. 
We can restrict the integration in \eqref{eq:integral algebraic lemma'}
to $\theta$ satisfying \eqref{eq: range of theta}. In this range   
\be\label{eq: xi and theta}
\absolute{\xi(\theta)-\xi(\theta_1)}\leq 2R\vare e^{-2t} \Lip(\xi).
\ee

Since $0\leq f_1, f_2\leq 1$ and $\xi$ is non-negative, we have 
\be\label{eq:integral algebraic lemma' 1}
\int_0^{2\pi} f(\darot v)\xi(\theta)\diff\!\theta\leq \int_0^{2\pi} f_1(a_tr_\theta)\xi(\theta)\diff\!\theta
\ee
Moreover, in view of~\eqref{eq: xi and theta}, we have 
\[
f_1(a_tr_\theta)\xi(\theta)=f_1(a_tr_\theta v_1)\xi(\theta_1)+O(R^2 \Lip(\xi)\vare e^{-2t})
\]
This,~\eqref{eq:integral algebraic lemma' 1} and the fact that the range of integration is \eqref{eq: range of theta} implies    
\[
e^{2t}\int_0^{2\pi} f(\darot v)\xi(\theta)\diff\!\theta\leq \xi(\theta_1)e^{2t}\int_0^{2\pi} f_1(a_tr_\theta v_1)+ O(R^2 \Lip(\xi)\vare e^{-2t}).
\]
This and~\eqref{eq: def J used} imply that 
\begin{multline}\label{eq:integral algebraic lemma''}
e^{2t}\int_0^{2\pi} f(\darot v)\xi(\theta)\diff\!\theta\leq\\ 
q^{-1}(1+O(\vare))J_{f_1}(e^{-t}\norm{v_1})\xi(\theta_1)+O(R^2 \Lip(f_1)\Lip(\xi)\vare e^{-2t})
\end{multline}
Thus~\eqref{eq:integral algebraic lemma'} follows from~\eqref{eq:integral algebraic lemma''} in view of~\eqref{eq: norm of v2} and \eqref{eq: value of sqf(v)}. 

Note that claim regarding $\mathcal E$ follows as well, indeed if either~\eqref{eq: norm of v1 and et},~\eqref{eq: norm of v2} or \eqref{eq: value of sqf(v)} fails, both the left and right side of~\eqref{eq:integral algebraic lemma'} equal zero.

\medskip

The proof of~\eqref{eq:integral algebraic lemma} is similar. Indeed one argues as in the proof of Lemma~\ref{lem: rotating to vertical 2} to show that if $e^{-t}\norm{v_2}\in I_1$ and $\sqf(v)\in[\alpha,\beta]$, then for all $\theta$ in the 
range~\eqref{eq: range of theta}, one has    
\[
f_2(a_tr_\theta v_2)=1.
\]
One then repeats the above argument and obtains~\eqref{eq:integral algebraic lemma}.   
\end{proof}


\subsection{A smooth cell decomposition}\label{sec: partition of unity}
Let 
\begin{align*}
\Omega&=\left\{(w_1+w_2, \omega(w_1-w_2)): \norm{w_k}\leq 1\right\},\\
\ball&=\left\{(v_1, v_2): \norm{v_k}\leq 1\right\}.
\end{align*}

As before, write $v=(v_1, v_2)\in\R^4$ where $v_k\in\R^2$.  
Let $\pi_1(v)=(v_1,0)$ and $\pi_2(v)=(0,v_2)$; abusing the notation, we also 
consider $\pi_k(\Omega)\subset\R^2$.  

Write $\Omega\setminus \ball=\Omega_1\cup\Omega_2$ where    
\begin{align*}
    \Omega_1:&=\Bigl\{(v_1, v_2)\in\Omega: \norm{v_1}> 1\Bigl\}\quad\text{and}\\
    \Omega_2:&=\Bigl\{(v_1, v_2)\in\Omega: \norm{v_1}\leq 1, \norm{v_2}>1\Bigl\}.
\end{align*}

A direct computation shows that $(v_1,v_2)\in\Omega$ if and only if
\[ \norm{v_2}^2\leq 4-\norm{v_1}^2-2\absolute{\sqf(v_1,v_2)}.\] 
It follows that 
for every $v_1\in \pi_1(\Omega_1)$, we have 
\be\label{eq: define I v1}
\{\norm{\lambda v_1}: (v_1, \lambda v_1)\in \Omega_1\}=\Bigl[0,\sqrt{4-\|v_1\|^2}\Bigr],
\ee
and for $v_2\in \pi_2(\Omega_2)$, we have 
\[
\{\norm{\lambda v_2}: (\lambda v_2, v_2)\in \Omega_2\}=\Bigl[0,\min\Bigl(1,\sqrt{4-\|v_2\|^2}\Bigr)\Bigr].
\]

\medskip

Fix some $R\geq 10^3$ and let $0<\vare< R^{-20}$. Let $E\in\N$ be so that $\frac1E\leq 100R^{10}\vare\leq \frac1{E-1}$, and put 
\[
I_i=[\tfrac{i-1}{E}, \tfrac{i}{E}]\quad\text{for all $1\leq i\leq E$.}
\]
Fix two families of smooth functions $\{\xi_i^-\}$ and $\{\xi_i^+\}$ with $C^1$ norm $\ll \vare^{-10}$ satisfying the following: 
\begin{enumerate}[label={($\xi$-\arabic*)}]
\item\label{xi-1} For all $i$, $0\leq\xi_i^-\leq \xi^+_i\leq 1$, 
\[
\begin{aligned}
&\xi_i^+ \equiv 1 \text{ on $2\pi I_i$,} && \supp(\xi_i^+)\subset 2\pi (I_i)_{\vare^2},\\
&\xi_i^- \equiv 1 \text{ on $2\pi(I_i)_{-4\vare^2}$,} && \supp(\xi_i^-)\subset 2\pi (I_i)_{-2\vare^2}
\end{aligned}
\] 
\end{enumerate} 
\noindent
(here we use the notation \eqref{eq:modified intervals}). We extend $\xi_i^\pm$ to $2\pi$-periodic functions on~$\R$.

\medskip

Similarly, let $E'\in\N$ be so that $\frac1{E'}\leq 100R^{9}\vare\leq \frac1{E'-1}$, and let 
\[
I_j'=[\tfrac{j-1}{E'}, \tfrac{j}{E'}]\quad\text{for all $1\leq j\leq E'$.}
\]
Fix two families of functions $\{\varrho_j^+\}$ and $\{\varrho_j^-\}$ with $C^1$ norm $\ll \vare^{-10}$ so that 
\begin{enumerate}[label={($\varrho$-\arabic*)}]
\item\label{rho-1} For all $i$, $0\leq\varrho_j^-\leq \varrho^+_j\leq 1$, 
\[
\begin{aligned}
&\varrho_j^+ \equiv 1 \text{ on $R I_j$,} && \supp(\varrho_j^+)\subset R (I_i)_{\vare^2},\\
&\varrho_j^- \equiv 1 \text{ on $R(I_j)_{-4\vare^2}$,} && \supp(\varrho_j^-)\subset R (I_i)_{-2\vare^2}.
\end{aligned}
\] 
\end{enumerate} 
Extend $\varrho_j^\pm$ to $\R$ by defining them to equal $0$ outside their supports.

Define \[
\varphi_{i,j}^+(\theta,r)=\xi_i^+(\theta)\varrho_j^+(r)\quad\text{and}\quad \varphi_{i,j}^-(\theta,r)=\xi_i^-(\theta)\varrho_j^-(r).
\] 
We will consider $\varphi_{i,j}^\pm$ as functions on $\R^2$ using our slightly non-standard polar coordinate system where  
any $0\neq w\in\R^2$ corresponds to $(\theta_w, \norm{w})$ if $r_{\theta_w}w=(0,\norm{w})$. Let 
\be\label{eq: def I Omega'}
\begin{aligned}
\mathcal I^+_1&=\{(i,j): \supp(\varphi_{i,j}^+)\cap \pi_1(\Omega_1)\neq \emptyset\}\\
\mathcal I^-_1&=\{(i,j): \supp(\varphi_{i,j}^-)\subset \pi_1(\Omega_1)\}.
\end{aligned}
\ee
We define $\mathcal I^\pm_2$ similarly with $\Omega_2$ and $\pi_2$ in lieu of $\Omega_1$ and $\pi_1$.
Note that for $k=1,2$ and $\boldsymbol\sigma=\pm$
\[
\biggl| {\rm area}(\pi_k(\Omega_k))-\sum_{(i,j)\in \mathcal I^{\boldsymbol\sigma}_k}\int\varphi_{i,j}^{\boldsymbol\sigma}\biggr|\ll \vare.
\]

We will work with $k=1$ for the remainder of this section, similar analysis applies to $k=2$ with the role of $v_1$ and $v_2$ switched. For all $(i,j)\in\mathcal I_1^+$, let 
\[
\Omega_{i,j}^+=\{(v_1,v_2+w): (v_1, v_2)\in\Omega_1,\varphi_{i,j}^+(v_1)=1, \norm{w}\leq 3R\vare\}.
\]

We will also define $\Omega_{i,j}\subset \Omega_{i,j}^+$ as follows. In view of~\eqref{eq: define I v1}, we will call the pair $(i,j)$ {\em typical} if 
\[
\inf\Bigl\{\sqrt{4-\|v_1\|^2}: v_1\in \supp(\varphi_{i,j}^+)\cap\pi_1(\Omega_1)\Bigr\}\geq \sqrt\vare.
\]
Let $\mathring{\mathcal I}_{1}^{-}$ denote the set of $(i,j)\in \mathcal I^-_1$ where 
$(i,j)$ is typical and for every 
\[
(v_1,\lambda v_1)\in\Omega_1\cap \Bigl(\supp(\varphi_{i,j}^-)\times \R^2\Bigr)\;\;\text{with}\;\; \norm{\lambda v_1} \in \left([0,\sqrt{4-\|v_1\|^2}]\right)_{-20 R\vare} 
\]
we have $(v_1, \lambda v_1+w)\in \Omega_1$ for all $w\in \R^2$ with $\norm{w}\leq 10R\vare$.

For any $(i,j)\in\mathring{\mathcal I}_{1}^{-}$, set 
\be\label{eq: def Omega ij 2}
\Omega_{i,j}:=\left\{(v_1,v_2+w): \begin{array}{c}(v_1, v_2)\in \Omega_1\cap \Bigl(\supp(\varphi_{i,j}^-)\times \R^2\Bigr),\\ 
w\in\R^2, \norm{w}\leq \vare\end{array}\right\}\cap \Omega_1.
\ee
Since $\supp(\varphi_{i,j}^-)\subset\{w: \varphi_{i,j}^+(w)=1\}$, we have $\Omega_{i,j}\subset \Omega_{i,j}^+$. Moreover, since $\{\supp(\varphi_{i,j}^-)\}$ is a disjoint collection, $\{\Omega_{i,j}\}$ is a disjoint collection. 

\medskip

In view of~\ref{xi-1}, \ref{rho-1}, and the above definitions,  
\begin{subequations}
\begin{align}
\label{eq: Omega ij 1} 1_{\Omega_1}\leq &\sum_{\mathcal I_1^+} 1_{\Omega_{i,j}^+}\leq 4\cdot 1_{\{(v_!, v_2): \norm{v_k}\leq 3\}} \\
\label{eq: Omega ij 2} &\sum_{\mathring{\mathcal I}_{1}^{-}} 1_{\Omega_{i,j}}\leq 1_{\Omega_1}
\end{align}
\end{subequations}

\subsection*{The intervals $I_{i,j}^+$ and $I_{i,j}^-$}
In our application of Lemma~\ref{lem: rotating to vertical}, $\xi_i^\pm$ will play the role of $\xi$; we will also work with $f=f_1f_2$ 
where $f_1$ is defined using $\varrho_j^\pm$ above and $f_2$ is defined using $I_0=[-q^{-1}\beta, -q^{-1}\alpha]$ (for some $R^{-1}\leq q\leq R$) and intervals $I_{i,j}^\pm$ which we now define. Put
\be\label{eq: def I ij pm}
\begin{aligned}
I_{i,j,+}'&=[0,b_{i,j}^+], \;\;\text{$b_{i,j}^+=\sup\Bigl\{\sqrt{4-\|v_1\|^2}\!:\! v_1\in \supp(\varphi_{i,j}^+)\cap\pi_1(\Omega_1)\Bigr\}$}\\
I_{i,j,-}'&=[0,b_{i,j}^-], \;\;\text{ $b_{i,j}^-=\inf\Bigl\{\sqrt{4-\|v_1\|^2}\!:\! v_1\in \supp(\varphi_{i,j}^+)\cap\pi_1(\Omega_1)\Bigr\}$}.
\end{aligned}
\ee

If $(i,j)$ is typical, i.e., if $b_{i,j}^-\geq\sqrt\vare$, put  
\be\label{eq: def I ij -}
I_{i,j}^+= \Bigl(I'_{i,j,+}\Bigr)_{10\vare} \quad\text{and} \quad I_{i,j}^-= \Bigl(I'_{i,j,-}\Bigr)_{-200R^{10}\vare}. 
\ee 
Since $\supp(\varphi_{i,j}^\pm)$ has diameter $\leq 200R^{10}\vare$ and $\vare<R^{-20}$, if $(i,j)$ is not typical, then $b_{i,j}^+\leq 2\sqrt\vare$. In this case, put $I_{i,j}^\pm=[0,3\sqrt\vare]$.

\medskip

We have the following lemma.

\begin{lemma}\label{lem: partition of unity}
Assume $R\geq \max\{10^3,\absolute{\alpha}, \absolute{\beta}\}$ and let $R^{-1}\leq q\leq R$.
Let $t \geq \log (R^2\vare^{-1})$, where as before $0<\vare<R^{-20}$.  

\begin{enumerate}
\item Let $I_0=[-q^{-1}\beta,-q^{-1}\alpha]$. Let $(i,j)\in\mathring{\mathcal I}_{1}^{-}$ and let  $f_1$ satisfy~\eqref{eq: def f1} with $\varrho_j^{-}$ (and with 
$\vare'=200R^{10}\vare$ instead of $\vare$). If 
\[
J_{f_{1}}(e^{-t}\norm{v_1})\xi_i^-(\theta_{v_1})1_{I_0^{(3)}}\Bigl(-q^{-1}\sqf(v)\Bigr) 1_{(I_{i,j}^-)^{(3)}}\Bigl(e^{-t}\norm{v_2}\Bigr)\neq0.
\]
for some $v=(v_1, v_2)\in\R^4$, then all the following hold 
\begin{enumerate}
\item $\sqf(v)\in ([\alpha,\beta])_{30R^4\vare}$, and
\item $e^{-t}v_1\in\supp(\varphi_{i,j}^-)$, and  
\item $e^{-t}v\in \Omega_{i,j}$. 
\end{enumerate}
\item Let $(i,j)\in\mathcal I_1^+$. If $v=(v_1,v_2)\in e^t\Omega_{i,j}^+$ satisfies $\sqf(v)\in [\alpha,\beta]$, then 
\[
e^{-t}\norm{v_2}\in I_{i,j}^+
\]
\end{enumerate}
\end{lemma}

\begin{proof}
We first prove part~(1).  
If $\sqf(v)\not\in ([\alpha,\beta])_{30R^4\vare}$, then \[-q^{-1}\sqf(v)\not\in (I_0)_{30R^3\vare}=I_0^{(3)},\]
hence 
\[
1_{I_0^{(3)}}\Bigl(-q^{-1}\sqf(v)\Bigr)=0.
\] 
Moreover, if we put $\bar v_1:=e^{-t}v_1$, then $\theta_{\bar v_1}=\theta_{v_1}$, and
$\bar v_1\not\in\supp(\varphi_{i,j}^-)$ would imply that $\varrho_j^-(e^{-t}\norm{v_1})\xi_i^-(\theta_{v_1})=0$. This in turn yields  
\[
0\leq f_1(\cox, e^{-t}\norm{v_1})\xi_i^-(\theta_{v_1})\leq \varrho_j^-(e^{-t}\norm{v_1})\xi_i^-(\theta_{v_1})=0,
\]
see~\eqref{eq: def f1}; thus, $J_{f_1}(e^{-t}\norm{v_1})\xi_i^-(\theta_{v_1})=0$. 
In conclusion, we may assume that 
\be\label{eq: e-tv2 in I}
J_{f_1}(e^{-t}\norm{v_1})\xi_i^-(\theta_{v_1})1_{I_0^{(3)}}\Bigl(-q^{-1}\sqf(v)\Bigr) 1_{(I_{i,j}^-)^{(3)}}\Bigl(e^{-t}\norm{v_2}\Bigr)\neq 0,
\ee
and that 
\[
\sqf(v)\in ([\alpha,\beta])_{30R^4\vare}\quad\text{and}\quad\bar v_1\in\supp(\varphi_{i,j}^-).
\] 
We need to show that (c) is also satisfied. 

Since $\sqf(v)\in ([\alpha,\beta])_{30R^4\vare}$, where $R\geq \max\{10^3,\absolute{\alpha},\absolute{\beta}\}$ and $\vare<R^{-20}$, 
and $\norm{v_1}\geq e^t$, there is $\lambda\in\R$ so that 
\be\label{eq: v2 = lambda v1 + w}
v_2=\lambda v_1+ w,\quad\text{where } w\perp v_1\text{ and } \norm{w}\leq 2R \norm{v_1}^{-1}\leq 2Re^{-t}.
\ee
Thus $e^{-t}v_2=\lambda e^{-t}v_1+e^{-t}w=\lambda \bar v_1+e^{-t}w$. 

Moreover, by~\eqref{eq: e-tv2 in I}, we have $e^{-t}\norm{v_2}\in (I_{i,j}^-)^{(3)}=(I_{i,j}^-)_{30R^3\vare}$, where 
\[
I_{i,j}^-= (I'_{i,j,-})_{-200R^{10}\vare}\quad\text{and}\quad I'_{i,j,-}\subset [0, \sqrt{4-\|\bar v_1\|^2}], 
\] 
see~\eqref{eq: def I ij pm} and~\eqref{eq: def I ij -}. 
Since $\norm{e^{-t}w}\leq 2R e^{-2t}$, 
we conclude that $\norm{\lambda\bar v_1}\in \Bigl([0, \sqrt{4-\|\bar v_1\|^2}]\Bigr)_{-20R\vare}$. In particular, 
\[
(\bar v_1, \lambda \bar v_1)\in \Omega_1\cap \Bigl(\supp(\varphi_{i,j}^-)\times \R^2\Bigr), 
\]
and $v=e^t(\bar v_1, \lambda \bar v_1+e^{-t}w)$ where $\norm{e^{-t}w}\leq 2Re^{-2t}$. 
By the definition of $\mathring{\mathcal I}_{1}^{-}$ and $\Omega_{i,j}$, we conclude that $e^{-t}v\in\Omega_{i,j}$. 
Thus, (c) also holds.  

The proof of~(2) is similar to the proof of (c), see in particular~\eqref{eq: v2 = lambda v1 + w}.  
\end{proof}


\subsection{Upper bound estimates}
Before starting the proof of Theorem~\ref{thm: main}, we record a weaker (but more explicit) version of 
\cite[Thm.\ 2.3]{EMM-Upp}, which will be used in the 
sequel --- see also the very recent work of Kelmer, Kontorovich, and Lutsko~\cite{KKL-Mean}.

For every $R>0$, let 
\[
\ball(R)=\{(v_1, v_2): \norm{v_k}\leq R\}.
\] 
Then $\ball(R)\setminus \ball(e^{-1}{R})=\ball(R)_1\cup\ball(R)_2$, where 
\begin{align*}
\ball_1(R)&=\{(v_1, v_2)\in\ball(R): e^{-1}R<\norm{v_1}\leq R\}\quad\text{and}\\
\ball_2(R)&=\{(v_1, v_2)\in\ball(R): \norm{v_1}\leq e^{-1}R, e^{-1}R<\norm{v_2}\leq R\}.
\end{align*}
We constructed smooth cell decomposition for $\Omega_1$ and $\Omega_2$ in \S\ref{sec: partition of unity}; in the following lemma we will use a similar construction (without repeating this construction) for $\ball_1(R)$ and $\ball_2(R)$.

  \begin{lemma}\label{lem: gen upp bd}
  Let $g=(g_1,g_2)\in G$ and put
  $\Lambda'=g\Lambda$.
 Let 
 \[
 R\geq \max\{10^3,\absolute{\alpha}, \absolute{\beta}, \norm{g_1}^{\pm1}, \norm{g_2}^{\pm1}\},
 \]
 and let $0<\eta<1$. 
 There exists $t_0\ll\absolute{\log\eta}$ so that if $t\geq t_0$, then 
  \[
  \#\{v=(v_1,v_2)\in \Lambda': \max(\norm{v_1},\norm{v_2})\leq Re^t, \alpha\leq \sqf(v)\leq \beta\}\ll e^{(2+\eta)t}
  \]
  where the implied constant depends polynomially on $R$.  
  \end{lemma}

  \begin{proof}
  The following basic lattice point estimate will be used:
\be\label{eq: contribution of vectors < et 1}
\#\{v\in \Lambda'\cap e^{t/2}\ball(R)\}\ll e^{2t}
\ee
where the implied constant depends polynomially on $R$.  

Since $R$ is fixed, we will denote $\ball_k(R)$ by $\ball_k$ ($k=1,2$) for the rest of the proof. Let $\vare= 10^{-6}R^{-20}$. 
Apply the construction in \S\ref{sec: partition of unity} for $\pi_1(\ball_1)$ with this $R$ and $\vare$. In particular, the functions $\xi_i^+$ are defined as in \ref{xi-1} with 
\[
I_i=[\tfrac{i-1}{E}, \tfrac{i}{E}]\quad\text{for all $1\leq i\leq E$ where $\frac1E\leq 100R^{10}\vare\leq \frac1{E-1}$.}
\]
and $\varrho_j^+$ are defined as in \ref{rho-1} with 
\[
I_j'=[\tfrac{j-1}{E'}, \tfrac{j}{E'}]\quad\text{for all $1\leq j\leq E'$ where $\frac1{E'}\leq 100R^{9}\vare\leq \frac1{E'-1}$.}
\]
For all $i,j$ as above, let $\xi_i=\xi^+_i$, $\varrho_j=\varrho_j^+$, and let $\varphi_{i,j}=\xi_i\varrho_j$. Put 
\[
\mathcal I_1^+=\{(i,j): \supp(\varphi_{i,j})\cap\pi_1(\ball_1)\neq \emptyset\};
\]
for all $(i,j)\in\mathcal I^+_1$, we have $\supp(\varrho_j)\subset [e^{-2}R, R]\subset [R^{-1}, R]$.

  For all $(i,j)\in\mathcal I_1^+$, put 
  \[
  \hat\ball_{i,j}=\{(v_1, v_2)\in\R^4: \varphi_{i,j}(v_1)=1, \norm{v_2}\leq R\}.
  \]
  Then $1_{\ball_1}\leq \sum_{\mathcal I_1^+}1_{\hat\ball_{i,j}}\leq 4_{\ball(2R)_1}$. 
   
  Define $f_1$ as in~\eqref{eq: def f1} for $q$ and $\varrho_j$, and with $200R^{10}\vare$ instead of $\vare$. Let
\be\label{eq: def f2 gen upp bd}
f_{2}=f_{[-q^{-1}\beta,-q^{-1}\alpha], [0, R]},
\ee
see~\eqref{eq: def f2}. Put $f_{i,j}=f_1f_{2}$. By the choice of $R$, we have $\sum f_{i,j}\leq 4_{\ball(2R)}$.

By Lemma~\ref{lem: rotating to vertical}, for any $v=(v_1, v_2)\in e^t\hat\ball_{i,j}$ with $\sqf(v)\in[\alpha,\beta]$, we have 
\begin{multline}\label{eq: gen upp bd 1}
e^{2t}\int_0^{2\pi} f_{i,j}(\darot v)\xi_i(\theta)\diff\!\theta=\\
q^{-1}(1+O(\vare))J_{f_1}\Bigl(e^{-t}\norm{v_1}\Bigr)\xi_i(\theta_{v_1}) f_{2}\Bigl(-q^{-1}\sqf(v), e^{-t}\norm{v_2}\Bigr)\\
+ O(\Lip(f_1)\Lip(\xi_i)e^{-2t})
\end{multline}
where the implied constant depends on $R$. 

First note that, if $t$ is large enough compared to $R$, we have  
\be\label{eq: gen upp bd error}
O(\Lip(f_1)\Lip(\xi_i)e^{-2t})\ll \vare^{-20}e^{-2t}\leq \vare^2.
\ee
Furthermore, for any $v=(v_1,v_2)\in e^t\hat\ball_{i,j}$ so that $Q_0(v)\in[\alpha, \beta]$, we have 
$f_{2}\Bigl(q^{-1}\sqf(v), e^{-t}\norm{v_2}\Bigr)=1$. 
Thus, using~\eqref{eq: Jf1}, we have  
\be\label{eq: gen upp bd 2}
J_{f_1}\Bigl(e^{-t}\norm{v_1}\Bigr)\xi_i(\theta_{v_1}) f_{2}\Bigl(q^{-1}\sqf(v), e^{-t}\norm{v_2}\Bigr)=\vare+O(\vare^2).
\ee

Put $x=g\Gamma'$. 
Summing~\eqref{eq: gen upp bd 1}, over all $v\in \Lambda'\cap e^t\hat\ball_{i,j}$ so that $Q_0(v)\in[\alpha, \beta]$
and using~\eqref{eq: gen upp bd 1} and~\eqref{eq: gen upp bd 2}, we conclude that 
\be\label{eq: gen upp bd main}
\vare\Bigl(\#\{v\in \Lambda'\cap e^t\hat\ball_{i,j}: \alpha\leq \sqf(v)\leq \beta\}\Bigr)\!\ll\! qe^{2t}\int_0^{2\pi} \hat f_{i,j}(\darot x)\diff\!\theta,
\ee 
where we used $0\leq \xi_i\leq 1$ and replaced $\vare^2+\vare+O(\vare^2)$ obtained from adding~\eqref{eq: gen upp bd error} and~\eqref{eq: gen upp bd 2} by $O(\vare)$.  

Summing~\eqref{eq: gen upp bd main} over all $(i,j)\in\mathcal I_1^+$ 
and using $\sum_{i,j} f_{i,j}\leq 4_{\ball(2R)}$, we get 
\[
\#\{v\in \Lambda'\cap e^t\ball_1: \alpha\leq \sqf(v)\leq \beta\}\ll \vare^{-1}qe^{2t}\int_0^{2\pi} \hat 1_{\ball(2R)}(\darot x)\diff\!\theta.
\]
One obtains a similar bound for the number $v\in\Lambda'\cap e^t\ball_2$ with $\sqf(v)\in[\alpha,\beta]$. Since $\ball\setminus e^{-1}\ball=\ball_1\cup\ball_2$ and $\vare=10^{-6}R^{-20}$, we conclude   
\begin{multline*}
    \#\{v\in \Lambda'\cap e^t(\ball\setminus e^{-1}\ball): \alpha\leq \sqf(v)\leq \beta\}\ll \\ \vare^{-1} qe^{2t}\int_0^{2\pi} \hat 1_{\ball(2R)}(\darot x)\diff\!\theta.
\end{multline*}
Let $t_{\eta}$ be as in Lemma~\ref{lem: gen bd for int hat f} applied with $\eta$ and $2R$, and let $t>10 t_{\eta}$. 
Then by Lemma~\ref{lem: gen bd for int hat f}, 
\[
\#\{v\in \Lambda'\cap e^t(\ball\setminus e^{-1}\ball): \alpha\leq \sqf(v)\leq \beta\}\ll e^{(2+\eta)t}.
\]
We may repeat the above with 
$t-\ell$ for all $0\leq \ell\leq t/2$, and obtain 
\be\label{eq: gen upp bd main 1}
\#\{v\in \Lambda'\cap e^{t-\ell}(\ball\setminus e^{-1}\ball): \alpha\leq \sqf(v)\leq \beta\}\ll e^{(2+\eta)(t-\ell)},
\ee
we also used $t-\ell\geq t/2\geq t_{\eta}$ when applying Lemma~\ref{lem: gen bd for int hat f} with $t-\ell$.  

Since $e^t(e^{-\ell}\ball)=e^{t-\ell}\ball$, summing~\eqref{eq: gen upp bd main 1} over $0\leq \ell\leq t/2$, we conclude  
\be\label{eq: gen upp bd main 2}
\#\{v\in \Lambda'\cap e^t(\ball\setminus e^{-t/2}\ball): \alpha\leq \sqf(v)\leq \beta\}\ll e^{(2+\eta)t}.
\ee
The lemma follows from~\eqref{eq: gen upp bd main 2} and~\eqref{eq: contribution of vectors < et 1}. 
\end{proof}


\section{Proof of Theorem~\ref{thm: main}}\label{sec: proof of main}
The proof relies on Theorem~\ref{thm: equi of f hat} and will be completed in some steps. 
Recall that $\torus=\R^2/\Delta$ and that $\Delta^*$ denotes the dual lattice. 
In view of our normalization, $2\pi\Delta^*=g_\torus \Z^2$ where $g_\torus\in\SL_2(\R)$. 
Let 
\be\label{eq: def g}
g=(g_\torus, -\omega g_\torus\omega)=(g_1,g_2)\in G \qquad\text{where $\omega=\begin{pmatrix}0 & -1\\ 1& 0\end{pmatrix}$.}
\ee
\subsection{Passage to $\sqf$}\label{sec: pass to sqf}
As it was observed in~\eqref{eq: B and Q0}, if $\lambda_i=\norm{v_i}^2$, where for $i=1,2$, $v_i\in 2\pi\Delta^*$  
is an eigenvalue of the Laplacian of $\torus$, then 
\be\label{eq: B and Q0'}
\lambda_1-\lambda_2=\sqf(v_1+v_2, \omega(v_1-v_2)).
\ee

Define $\Omega=\{(v_1+v_2, \omega(v_1-v_2)): \norm{v_i}\leq 1\}$; 
and let 
\[
\Lambda'=\{(v_1+v_2, \omega(v_1-v_2)): v_1,v_2\in 2\pi\Delta^*\}=g\Lambda
\] 
where $\Lambda = \{(v_1+v_2, \omega(v_1-v_2)): v_1,v_2\in \Z^2\}$.

Let $T$ be a (large) parameter, and put $t=\frac12\log T$. 
In view of~\eqref{eq: B and Q0'},  
\be\label{eq: counting Omega'}
R_\torus(\alpha, \beta, T)=\#\{v\in \Lambda'_{\rm nz}\cap e^t\Omega:  \alpha\leq \sqf(v)\leq \beta\};
\ee
recall that $\Lambda'_{\rm nz}=\{(w_1, w_2)\in\Lambda': w_i\neq0\}$. 

Let $A$  and $\delta$ be as in Theorem~\ref{thm: main}. 
Without loss of generality, we assume $A\geq 10^5$ and $0<\delta<10^{-5}$. 
Let $\hat A$ be given by Theorem~\ref{thm: equi of f hat} applied with $10^3A$. 
We will show the claim in Theorem~\ref{thm: main} holds with $A'=10\hat A$. 
To simplify the notation, write $\bar A=10^3A$ for the rest of the proof.

Thus let us assume~\eqref{eq: Diophantine condition} holds for $A'$:
for $T\geq T_0$ ($T_0$ is a yet to be determined large constant) and all $(p_1, p_2, q)\in\Z^3$ with $T^{\delta/A'}<q<T^{\delta}$, 
\be\label{eq: Diophantine condition'}
\Bigl| \tfrac{\mathsf b}{\mathsf a} -\tfrac{p_1}{q}\Bigr|+ \Bigl|\tfrac{\mathsf c}{\mathsf a}-\tfrac{p_2}{q}\Bigr|> q^{-A}.
\ee
This implies that so long as $t=\frac12\log T$ is large enough (depending on $\mathsf a$, $\mathsf b$, and $\mathsf c$), we have 
\be\label{eq: g2-1 g1 proof}
g_2^{-1}g_1=-\omega g_\torus^{-1}\omega g_\torus= \begin{pmatrix}\mathsf a & \mathsf b\\ \mathsf b &\mathsf c\end{pmatrix}
\ee
satisfies~\eqref{eq: Diophantine main them Q0} with $t$, $\rho=\delta/10$, $\hat A$. That is:
for every $Q\in\Mat_2(\Z)$ with $e^{\rho t/\hat A}\leq \norm{Q}\leq e^{\rho t}$ and all $\lambda\in\R$, we have  
\be\label{eq: Diophantine main them Q0'}
\norm{g_2^{-1}g_1-{\lambda}Q}>\norm{Q}^{-A}=\norm{Q}^{-\bar A/1000}.
\ee

\begin{lemma}\label{lem: exceptional subspaces t-ell}
    There are at most two $g\Z^4$-rational two dimensional subspaces 
    $L,L'$ so that if for some $2t/5\leq s\leq t$, $L_s$ is a
    $(\delta_1/\bar A, \delta_1, s)$-exceptional subspace, then $L_s=L$ or $L'$.
\end{lemma}

\begin{proof}
Let $2t/5\leq s\leq t$. Recall that a $(\delta_1/\bar A, \delta_1, s)$-exceptional subspace is spanned by two vectors $(g_1w_1, 0), (0,g_2w_2)\in g\Z^4$ satisfying  
\be\label{eq: special proof main}
\begin{aligned}
&0<\norm{g_iw_i}\leq e^{\delta_1 s/\bar A}, \quad\text{and}\\ 
&\absolute{\sqf(g_1w_1, g_2w_2)}\leq e^{-\delta_1 s}.
\end{aligned}
\ee
We also note that 
\[
e^{\delta_1 s/\bar A}\leq e^{\delta_1t/\bar A}\quad\text{and}\quad e^{-\delta_1 s}\leq e^{-2\delta_1t/5}
\] 
for any $2t/5\leq s\leq t$. 

Assume now that there are three pairs (possibly corresponding to different values of $2t/5\leq s\leq t$) so that~\eqref{eq: special proof main} is satisfied. Then    
Lemma~\ref{lem: special subspace}, applied with $\delta_1/\bar A$ and $2\bar A/5$, implies that there is $Q\in\Mat_2(\Z)$ with $\norm{Q}\leq e^{100\delta_1t/\bar A}$ so that  
\begin{multline*}
\norm{g_2^{-1}g_1-{\lambda}Q}=\norm{\begin{pmatrix}\mathsf a & \mathsf b\\ \mathsf b &\mathsf c\end{pmatrix}-{\lambda}Q}\leq e^{-(\frac{2\bar A}5-100)(\delta_1/\bar A)}\leq\\ \max\Bigl\{\norm{Q}^{-\bar A/1000}, 100e^{-\rho\bar A t/(1000\hat A)}\Bigr\}.
\end{multline*}
Since $\rho/\hat A\leq \delta_1/\bar A\leq \rho/100$, this contradicts the fact that 
$g_2^{-1}g_1$ satisfies~\eqref{eq: Diophantine main them Q0'} with $t$, $\rho$, $\hat A$ --- note that if $\norm{Q}\leq e^{\rho t/\hat A}$, we may replace $Q$ by an integral multiple $nQ$ with $e^{\rho/\hat A}\leq \norm{nQ}\leq 2e^{\rho t/\hat A}$ . The proof is complete.
\end{proof}

Let $L$ and $L'$ be as in Lemma~\ref{lem: exceptional subspaces t-ell}. 
For a set $\mathsf E\subset \R^4$ and $s>0$ we let 
\[
\begin{aligned}
N_s(\mathsf E)&:=\#\{v\in \Lambda'_{\rm nz}\cap e^s\mathsf E: \alpha\leq \sqf(v)\leq\beta\},\\
N_s'(\mathsf E)&:=\#\{v\in (\Lambda'_{\rm nz}\setminus (L_s\cup L'_s)) \cap e^s\mathsf E: \alpha\leq \sqf(v)\leq\beta\}.
\end{aligned}
\]


\subsection{Counting and circular averages}\label{sec: counting and circular}
For the rest of the proof, we fix $\vare=e^{-\eta' t}$ for some $0<\eta'<1/100$ which is small and will be optimized later. We will also assume $\beta-\alpha\geq \vare$ otherwise Theorem~\ref{thm: main} holds trivially. 

Recall that   
\[
\Omega=\{(w_1+w_2, \omega(w_1-w_2)): \norm{w_i}\leq 1\},
\]
and that $\Omega\setminus\ball=\Omega_1\cup \Omega_2$ where $\ball=\{(v_1,v_2)\in\R^4: \norm{v_k}\leq 1\}$, and 
\begin{align*}
    \Omega_1&=\Bigl\{(v_1, v_2)\in\Omega: \norm{v_1}> 1\Bigl\}\quad\text{and}\\
    \Omega_2&=\Bigl\{(v_1, v_2)\in\Omega: \norm{v_1}\leq 1, \norm{v_2}>1\Bigl\}.
\end{align*}

Let $R$ be a large constant (we will always assume $R<\vare^{-1/20}$, hence, $R$ is much smaller that $e^t$), satisfying  
\[
R\geq \max\{10^3, \absolute{\alpha}, \absolute{\beta},\absolute{\mathsf a}, \absolute{\mathsf b},\absolute{\mathsf c}\};
\]
note that $\pi_k(\Omega)\subset B(0,R)$.

Apply the construction in \S\ref{sec: partition of unity} for $\pi_k(\Omega_k)$ with $\vare$ and $R$ here. The analysis for $k=1$ and $2$ are similar, thus, let $k=1$ until further notice. Let 
\[
\varphi_{i,j}^\pm=\xi_i^\pm\varrho_j^\pm\qquad\text{ for }(i,j)\in\mathcal I_1^\pm.
\] 
Note that $\supp(\varrho_j^\pm)\subset [q-200R^{10}\vare, q]\subset [R^{-1},R]$ for some $R^{-1}\leq q\leq R$, see \ref{rho-1} --- indeed in the case at hand, we have $1\leq q\leq 2$.

For ${\boldsymbol\sigma}=\pm$, define $f_{1}^{\boldsymbol\sigma}$ as in~\eqref{eq: def f1}
for $q$ and $\varrho_j^{\boldsymbol\sigma}$. Let 
\be\label{eq: def f2 proof}
f_{2}^{\boldsymbol\sigma}=f_{I_{0}^{\boldsymbol\sigma}, I_{i,j}^{\boldsymbol\sigma}},
\ee
where $I_{0}^+=[-q^{-1}\beta,-q^{-1}\alpha]$ and $I_{0}^-=\Bigl(I_{0}^+\Bigr)_{-100R^5\vare}$, see~\eqref{eq: def f2} and~\eqref{eq:modified intervals}. 
Put 
\[
f_{i,j}^{\boldsymbol\sigma}=f_{1}^{\boldsymbol\sigma} f_{2}^{\boldsymbol\sigma}.
\]

\begin{lemma}\label{lem: upp and lower bd}
Let the notation be as above, and let $L$ and $L'$ denote $(\delta_1/\bar A, \delta_1, t)$-exceptional subspaces if they exist. 

If $(i,j)\in\mathring{\mathcal I}_1^-$, then  
\begin{multline}\label{eq: upp and lower bd phi ij}
qe^{2t}\sum_{v\in\Lambda'_{\rm nz}\setminus (L\cup L')}\int_0^{2\pi} f_{i,j}^{-}(\darot v)\xi_i^-(\theta)\diff\!\theta\leq \\
\qquad\qquad (\vare+O(\vare^2)) \cdot N_t'(\Omega_{i,j}) + O(\vare^{-21}).
\end{multline} 
Moreover for every $(i,j)\in\mathcal I_1^+$, we have 
\begin{multline}\label{eq: upp bd phi ij +}
(\vare+O(\vare^2))\cdot N_t'(\Omega_{i,j}^+)\leq 
qe^{2t}\!\!\!\sum_{v\in\Lambda'_{\rm nz}\setminus(L\cup L')}\int_0^{2\pi}\!\! f_{i,j}^{+}(\darot v)\xi_i^+(\theta)\diff\!\theta.
\end{multline}  
The implied constants depend polynomially on $R$. 
\end{lemma}

The proof is similar to the proof of Lemma~\ref{lem: gen upp bd}. 
More precisely, we will use~\eqref{eq:integral algebraic lemma'} for $f_{i,j}^{-}$ 
and~\eqref{eq:integral algebraic lemma} for $f_{i,j}^{+}$; let us now turn to the details.

\begin{proof}
When there is no confusion we drop $i,j$ from the notation and denote $f_{i,j}^\pm$ by $f^\pm$, $\xi_i^\pm$ by $\xi^\pm$, etc. 
Also, we will put $I_0=I_{0}^-$ and $I_1=I_{i,j}^-$, but will keep the more cumbersome notation for $I_{0}^+$ and $I_{i,j}^+$.  

By~\eqref{eq:integral algebraic lemma'} in Lemma~\ref{lem: rotating to vertical} applied with $f^-=f_{i,j}^-$, for any $v\in\R^4$, we have  
\begin{multline}\label{eq: lower bd 1'}
qe^{2t}\int_0^{2\pi} f^-(\darot v)\xi^-(\theta)\diff\!\theta\leq\\
(1+O(\vare))J_{f_{1}^-}\Bigl(e^{-t}\norm{v_1}\Bigr)\xi^-(\theta_{v_1}) 1_{I_0^{(3)}}\Bigl(-q^{-1}\sqf(v)\Bigr)1_{I_1^{(3)}}\Bigl(e^{-t}\norm{v_2}\Bigr)+ \mathcal E,
\end{multline}
where $I^{(k)}=I_{10kR^3\vare}$ and
\be\label{eq: def of Error}
\mathcal E=O\Bigl(\Lip(f_{1}^-)\Lip(\xi^-)e^{-2t}\Bigr) 
\ee
furthermore, $\mathcal E=0$ if 
\[
\Bigl(-q^{-1}\sqf(v), e^{-t\norm{v_2}}\Bigr)\not\in I_0^{(3)}\times I_1^{(3)}\quad\text{or} \quad\norm{v_1}> 2Re^t.
\]

By ~\eqref{eq:integral algebraic lemma} in Lemma~\ref{lem: rotating to vertical} applied with $f^+=f_{i,j}^+$, 
for any $v\in\R^4$ with $e^{-t}\norm{v_2}\in I_{i,j}^+$ and $\sqf(v)\in [\alpha,\beta]$, we have 
\begin{multline}\label{eq: lower bd 1}
qe^{2t}\int_0^{2\pi} f^{+}(\darot v)\xi^+(\theta)\diff\!\theta=\\
(1+O(\vare))J_{f_{1}^+}\Bigl(e^{-t}\norm{v_1}\Bigr)\xi^+(\theta_{v_1}) f_{2}^+\Bigl(-q^{-1}\sqf(v), e^{-t}\norm{v_2}\Bigr)\\
+ O\Bigl(\Lip(f_{1}^+)\Lip(\xi^+)e^{-2t}\Bigr).
\end{multline}
In particular,~\eqref{eq: lower bd 1} holds for all $v\in e^t\Omega_{i,j}^+$ with $\sqf(v)\in[\alpha,\beta]$ thanks to part~(2) in Lemma~\ref{lem: partition of unity}. 

\medskip

Before analysing~\eqref{eq: lower bd 1'} and~\eqref{eq: lower bd 1} further, we record the following:
\be\label{eq: third line in eq lower bd 1}
O\Bigl(\Lip(f_{1}^\pm)\Lip(\xi^\pm)e^{-2t}\Bigr)=O(\vare^{-20} e^{-2t})\ll \vare^3,
\ee
so long as $t$ is large enough (recall that the implied constants depend polynomially on $R$).

\medskip

Let us now begin with~\eqref{eq: lower bd 1}. 
In view of~\eqref{eq: Jf1}, 
for any $v=(v_1,v_2)\in e^t\Omega_1$ so that $\alpha\leq Q_0(v)\leq \beta$, we have
\begin{multline}\label{eq: lower bd 2}
J_{f_{1}^+}\Bigl(e^{-t}\norm{v_1}\Bigr)\xi^+(\theta_{v_1}) f_{2}^+\Bigl(-q^{-1}\sqf(v), e^{-t}\norm{v_2}\Bigr)=\\
(\vare+O(\vare^2)) \varrho^+(e^{-t}\norm{v_1})\xi^+(\theta_{v_1})f_{2}^+\Bigl(-q^{-1}\sqf(v), e^{-t}\norm{v_2}\Bigr).
\end{multline}
Moreover, for every $v\in e^t\Omega_{i,j}^+$, satisfying $\alpha\leq Q_0(v)\leq \beta$, 
\[
f_{2}^+\Bigl(-q^{-1}\sqf(v), e^{-t}\norm{v_2}\Bigr)=1, \;\;\xi^+(\theta_{v_1})=1, \;\;\text{and}\;\;  \varrho^+(e^{-t}\norm{v_1})=1;
\] 
from this and~\eqref{eq: lower bd 2}, we conclude that 
\[
J_{f_{1}^+}\Bigl(e^{-t}\norm{v_1}\Bigr)\xi^+(\theta_{v_1}) f_{2}^+\Bigl(-q^{-1}\sqf(v), e^{-t}\norm{v_2}\Bigr)= (\vare+O(\vare^2)). 
\] 
Together with~\eqref{eq: lower bd 1} and~\eqref{eq: third line in eq lower bd 1}, this implies that 
\be\label{eq: upp bd 1}
qe^{2t}\int_0^{2\pi} f^{+}(\darot v)\xi^+(\theta)\diff\!\theta=\vare+O(\vare^2) 
\ee
for every $v\in e^t\Omega_{i,j}^+$ with $\alpha\leq Q_0(v)\leq \beta$. 

Summing~\eqref{eq: upp bd 1}, over all such $v\in\Lambda'_{\rm nz}\setminus(L\cup L')$, we obtain
\be\label{eq: upp bd}
(\vare+O(\vare^2))\cdot N_t'(\Omega_{i,j}^+)\leq 
qe^{2t}\!\!\!\sum_{v\in\Lambda'_{\rm nz}\setminus (L\cup L')}\! \int_0^{2\pi} f^{+}(\darot v)\xi^+(\theta)\diff\!\theta.
\ee
This establishes~\eqref{eq: upp bd phi ij +}. 

Let us now assume $(i,j)\in\mathring{\mathcal I}_{1}^{-}$ and obtain a lower bound for $N_t(\Omega_{i,j})$.
For this, we investigate the term appearing in the second line of~\eqref{eq: lower bd 1'}.

We first claim that 
\[
J_{f_{1}^-}\Bigl(e^{-t}\norm{v_1}\Bigr)\xi^-(\theta_{v_1}) 1_{I_0^{(3)}}(-q^{-1}\sqf(v))1_{I_1^{(3)}}(e^{-t}\norm{v_2})\neq 0,
\]
then $\sqf(v)\in [\alpha, \beta]$ and $v\in e^t\Omega_{i,j}$. 

To see the claim, recall that by part~(1) in Lemma~\ref{lem: partition of unity}, for any $v\in\R^4$,  
\[
J_{f_{1}^-}\Bigl(e^{-t}\norm{v_1}\Bigr)\xi^-(\theta_{v_1}) 1_{I_0^{(3)}}(-q^{-1}\sqf(v))1_{I_1^{(3)}}(e^{-t}\norm{v_2})=0
\]
unless all the following are satisfied 
\begin{subequations}
\begin{align}
\label{eq: partition unity use 1}&\sqf(v)\in [\alpha+50R^5\vare, \beta-50R^5\vare],\\
\label{eq: partition unity use 2}&v_1\in e^{t}\supp(\varphi_{i,j}^-),\quad\text{and}\\   
\label{eq: partition unity use 3}&v\in e^t\Omega_{i,j}.
\end{align}
\end{subequations}
in deducing~\eqref{eq: partition unity use 1} from Lemma~\ref{lem: partition of unity}, we used the definitions 
\[
I_0^{(3)}=(I_{0}^-)_{30R^3\vare}\quad\text{and}\quad I_{0}^-=\Bigl([-q^{-1}\beta, -q^{-1}\alpha]\Bigr)_{-100R^5\vare}.
\]

We conclude from~\eqref{eq: partition unity use 1} that $\sqf(v)\in [\alpha, \beta]$. Using the definition of $\Omega_{i,j}$ in~\eqref{eq: def Omega ij 2} and since $2R^3e^{-2t}<\vare$, ~\eqref{eq: partition unity use 3} implies that $v\in e^{t}\Omega_{i,j}$, and completes the proof of the claim.

\medskip

We now return to the proof of the lemma. Recall that 
\begin{multline*}
J_{f_{1}^-}(e^{-t}\norm{v_1})\xi^-(\theta_{v_1})1_{I_0^{(3)}}(-q^{-1}\sqf(v))1_{I_1^{(3)}}(e^{-t}\norm{v_2})\leq \\J_{f_{1}^-}(e^{-t}\norm{v_1})
=(\vare+O(\vare^2)) \varrho^-(e^{-t}\norm{v_1})\leq \vare+O(\vare^2).
\end{multline*}
This and the above claim imply that 
\begin{multline}\label{eq: lower bd 2'} 
\sum_{v\in \Lambda'_{\rm nz}\setminus(L\cup L')} J_{f_{1}^-}(e^{-t}\norm{v_1})\xi^-(\theta_{v_1})1_{I_0^{(3)}}(-q^{-1}\sqf(v))1_{I_1^{(3)}}(e^{-t}\norm{v_2})\\\leq 
(\vare+O(\vare^2))\cdot N_t'(\Omega_{i,j}).
\end{multline}
Moreover, since $\Bigl(-q^{-1}\sqf(v), e^{-t\norm{v_2}}\Bigr)\not\in I_0^{(3)}\times I_1^{(3)}$ or $\norm{v_1}> 2Re^t$ imply $\mathcal E=0$. We conclude from Lemma~\ref{lem: gen upp bd} applied with $\eta=\eta'/10$ imply that
\[
\sum_{v\in\Lambda'} \mathcal E\ll \vare^{-20}e^{-2t} e^{(2+\eta)t}\ll \vare^{-21};
\]  
we used $\Lip(f_{1}^-)\Lip(\xi^-)e^{-2t}\ll\vare^{-20}e^{-2t}$, see~\eqref{eq: third line in eq lower bd 1}, 
and $\vare=e^{-\eta' t}$. 
This,~\eqref{eq: lower bd 2'} and~\eqref{eq: lower bd 1'} imply that
\begin{multline*}
qe^{2t}\sum_{v\in\Lambda'_{\rm nz}\setminus L\cup L'}\int_0^{2\pi} f^{-}(\darot v)\xi(\theta)\diff\!\theta +  O(\vare^{-21})\leq \\
(\vare+O(\vare^2))\cdot N_t'(\Omega_{i,j}),
\end{multline*} 
as we claimed in~\eqref{eq: upp and lower bd phi ij}. 
\end{proof}

We will use Theorem~\ref{thm: equi of f hat} to reduce both~\eqref{eq: upp and lower bd phi ij} and~\eqref{eq: upp bd phi ij +} 
to the study of $\int_X\hat f_{i,j}^\pm\diff\!m_X$, see~\eqref{eq: def hat f}. Let us begin with computing this integral.

\begin{lemma}\label{lem: int f hat over X}
For ${\boldsymbol\sigma}=\pm$ let $f_{i,j}^{{\boldsymbol\sigma}}=f_1^{\boldsymbol\sigma} f_2^{\boldsymbol\sigma}$, where for $k=1,2$, $f_k^{\boldsymbol\sigma}$ is as in~\S\ref{sec: counting and circular}.
There is an absolute constant $c_\Lambda$ so that
\be\label{eq: circ int 2}
q\int_X\hat f_{i,j}^{{\boldsymbol\sigma}}\diff\!m_X= c_\Lambda\vare (\beta-\alpha)\absolute{I_{i,j}^{\boldsymbol\sigma}}\int\varrho_j^{\boldsymbol\sigma}+O(\vare^2)(\beta-\alpha)\absolute{I_{i,j}^{\boldsymbol\sigma}}\int\varrho_j^{\boldsymbol\sigma}. 
\ee
\end{lemma}

\begin{proof}
We have 
\begin{align*}
\int_X\hat f_{i,j}^{{\boldsymbol\sigma}}\diff\!m_X&= c_\Lambda\int_{\R^2} f_1^{\boldsymbol\sigma}\int_{\R^2}f_2^{\boldsymbol\sigma}\\
&= c_\Lambda \vare \int_{\R}\varrho_j^{\boldsymbol\sigma}\int_{\R^2} f_2^{\boldsymbol\sigma}+ O(\vare^2) \int_{\R}\varrho_j^{\boldsymbol\sigma}\int_{\R^2} f_2^{\boldsymbol\sigma}
\end{align*}
where $c_\Lambda$ is absolute and the implied constants depend only on $R$.

Since $f_2$ is defined as in~\eqref{eq: def f2 proof}, we conclude that 
\[
\int f_2^{\boldsymbol\sigma}= q^{-1}(\beta-\alpha)|I_{i,j}^{\boldsymbol\sigma}|+ O\Bigl(q^{-1}\vare(\beta-\alpha)|I_{i,j}^{\boldsymbol\sigma}|\Bigr)
\]
again the implied constant depends only on $R$. The lemma follows. 
\end{proof}

\begin{lemma}\label{lem: upp and lower bd 2}
Let the notation be as in Lemma~\ref{lem: int f hat over X}. In particular, 
\[
f_{i,j}^{\pm}=f_1^\pm f_2^\pm,
\] 
where $f_k^\pm$ are as in~\S\ref{sec: counting and circular}. Also put 
\[
\Upsilon_{i,j}^\pm=c_\Lambda (\beta-\alpha)\absolute{I_{i,j}^\pm}\int \xi_i^\pm\int\varrho_j^\pm.
\] 

If $(i,j)\in\mathring{\mathcal I}_1^-$, then 
\be\label{eq: upp and lower bd phi ij 2}
e^{2t}\Bigl(\Upsilon_{i,j}^- + O\Bigl(\Sob(f_{i,j}^-)\Sob(\xi_i^-)e^{-\delta_2 t}\Bigr)\Bigr)\leq 
(1+O(\vare))\cdot N_t'(\Omega_{i,j})
\ee
Moreover, for every $(i,j)\in\mathcal I_1^+$, we have 
\be\label{eq: upp bd phi ij + 2}
(1+O(\vare))\cdot N_t'(\Omega_{i,j}^+)\leq e^{2t}\Bigl(\Upsilon_{i,j}^++O\Bigl(\Sob(f_{i,j}^+)\Sob(\xi_i^+)e^{-\delta_2 t}\Bigr)\Bigr).
\ee 
where the implied constants depends polynomially on $R$. 
\end{lemma}

\begin{proof}
We will prove the lemma using Lemma~\ref{lem: upp and lower bd} and Theorem~\ref{thm: equi of f hat}. Let us begin with restating the main conclusion of Theorem~\ref{thm: equi of f hat} in the form which will be used here. When there is no confusion, we drop $i,j$ from the notation and denote $f_{i,j}^\pm$ by $f^\pm$, $\xi_i^\pm$ by $\xi^\pm$, etc. 

Recall that $\Lambda'=g\Lambda$ where $g=(g_1, g_2)$ is as in~\eqref{eq: def g}. 
Let $L$ and $L'$ be as in Lemma~\ref{lem: exceptional subspaces t-ell} if they exist. For ${\boldsymbol\sigma}=\pm$, put 
\begin{align*}
\hat f_{{\rm sp}}^{\boldsymbol\sigma}(\theta)&=\!\!\!\sum_{v\in \Lambda' \cap (L\cup L')} f_{{\rm sp}}^{\boldsymbol\sigma}(\Delta(a_t r_\theta)v)\\
\mathcal C_{\boldsymbol\sigma}&=\left\{\theta\in[0,2\pi]: \hat f_{{\rm sp}}^{\boldsymbol\sigma}(\theta)\geq e^{\delta_1 t}\right\},
\end{align*}
and define 
\[
\hat f^{{\boldsymbol\sigma}}_{\rm mod}(\theta)=\begin{cases}\hat f^{{\boldsymbol\sigma}}(\theta)-\hat f_{{\rm sp}}^{\boldsymbol\sigma}(\theta)&\text{$\theta\in\mathcal C_{\boldsymbol\sigma}$}\\
\hat f^{{\boldsymbol\sigma}}(\theta) &\text{otherwise}\end{cases}
\]
where we write $\hat f^{{\boldsymbol\sigma}}(\theta)=\hat f^{{\boldsymbol\sigma}}(\darot g\Gamma')$. 

Since $g$ satisfies~\eqref{eq: Diophantine main them Q0'},
Theorem~\ref{thm: equi of f hat} and the definition of $\hat f^{{\boldsymbol\sigma}}_{\rm mod}(\theta)$ imply
\be\label{eq: circ int 1 1}
\int_0^{2\pi}\!\! \hat f^{{\boldsymbol\sigma}}_{\rm mod}(\theta)\xi^{\boldsymbol\sigma}(\theta)\diff\!\theta \!=\!\!
\int\xi^{\boldsymbol\sigma}\!\diff\!\theta\!\!\int_X\hat f^{{\boldsymbol\sigma}}\!\diff\!m_X+O\Bigl(\Sob(f^{{\boldsymbol\sigma}})\Sob(\xi^{\boldsymbol\sigma})e^{-\delta_2 t}\Bigr).
\ee

With this established, we first show~\eqref{eq: upp and lower bd phi ij 2}. Let $\boldsymbol\sigma=-$. 
Assuming $\eta'$ in the definition of $\vare=e^{-\eta't}$ is small enough, we have  
\[
O(\Sob(f^{-})\Sob(\xi^{-})e^{-\delta_2 t})< \vare^4(\beta-\alpha).
\]
Recall from \S\ref{sec: partition of unity} that $\int\varrho_j^{-}\geq \vare$ and that $\absolute{I_{i,j}^{-}}\geq \sqrt\vare$. 
Thus~\eqref{eq: circ int 1 1}, together with the above and Lemma~\ref{lem: int f hat over X}, implies that
\be\label{eq: circ int 1 2}
\int_0^{2\pi}\!\! \hat f^{-}_{\rm mod}(\theta)\xi^-(\theta)\diff\!\theta \gg \vare^3(\beta-\alpha).
\ee
Moreover, by part~(2) in Lemma~\ref{lem: gen bd for int hat f} applied with $\delta_1$, $L$, and $L'$, we have 
\be\label{eq: not cusp of sp subspaces}
\int_{[0,2\pi]\setminus\mathcal C_-}\hat f_{{\rm sp}}^-(\theta)\diff\!\theta\ll e^{(-1+\delta_1)t}
\ee
Recall that $\delta_1<1/100$, hence, if $\eta'<1/100$, then $e^{(-1+\delta_1)t}<\vare^4(\beta-\alpha)$. 
Thus, we get from~\eqref{eq: circ int 1 2} and~\eqref{eq: not cusp of sp subspaces}  
\be\label{eq: circ int 1 2 3}
\begin{aligned}
&\sum_{v\in\Lambda'_{\rm nz}\setminus (L\cup L')}\int_0^{2\pi} f^-(\darot v)\xi^-(\theta)\diff\!\theta\\
&=\int_0^{2\pi}\!\! \hat f^-_{\rm mod}(\theta)\xi^-(\theta)\diff\!\theta - \int_{[0,2\pi]\setminus\mathcal C_-}\hat f_{\rm sp}^-(\theta)\xi^-_i(\theta)\diff\!\theta \\
&=(1+O(\vare))\int_0^{2\pi}\!\! \hat f^-_{\rm mod}(\theta)\xi^-(\theta)\diff\!\theta.
\end{aligned}
\ee

In view of~\eqref{eq: upp and lower bd phi ij} in Lemma~\ref{lem: upp and lower bd}, 
\begin{multline*}
qe^{2t}\sum_{v\in\Lambda'_{\rm nz}\setminus (L\cup L')}\int_0^{2\pi} f^-(\darot v)\xi^-(\theta)\diff\!\theta +  O(\vare^{-21})\leq \\
(\vare+O(\vare^2))\cdot N_t'(\Omega_{i,j})
\end{multline*}
Using this and~\eqref{eq: circ int 1 2 3} (multiplied by $qe^{2t}$), we conclude 
\[
qe^{2t} (1+O(\vare))\int_0^{2\pi}\!\! \hat f^{-}_{\rm mod}(\theta)\xi^-(\theta)\diff\!\theta +  O(\vare^{-21})\leq 
(\vare+O(\vare^2))\cdot N_t'(\Omega_{i,j}).
\]
This,~\eqref{eq: circ int 1 1} and~\eqref{eq: circ int 2} yield, 
\be\label{eq: upp boud lemma used'} e^{2t}(
\Upsilon_{i,j}^-+ O(\Sob(f^{-})\Sob(\xi^-)e^{-\delta_2 t}) +O(\vare^{-21})\leq (1+O(\vare))\cdot N_t'(\Omega_{i,j}).
\ee
Assuming $\eta'$ is small enough and $t$ large, we have  
\[
\vare^{-23}<e^{2t}\cdot \biggl(c_\Lambda (\beta-\alpha)\absolute{I_{i,j, -}}\int_{\R}\varrho_j^-\int_\R\xi_i^-\biggr).
\]
Hence,~\eqref{eq: upp and lower bd phi ij 2} follows from~\eqref{eq: upp boud lemma used'}.

We now show~\eqref{eq: upp bd phi ij + 2}; the argument is similar and simpler. By~\eqref{eq: upp bd phi ij +},   
\begin{multline}\label{eq: upp bd phi ij + use}
(\vare+O(\vare^2))\cdot N_t'(\Omega_{i,j}^+)\leq 
qe^{2t}\!\!\!\sum_{v\in\Lambda'_{\rm nz}\setminus(L\cup L')}\int_0^{2\pi}\!\! f^+(\darot v)\xi^+(\theta)\diff\!\theta\\
\leq qe^{2t}\int_0^{2\pi}\!\! \hat f^+_{\rm mod}(\theta)\xi^+(\theta)\diff\!\theta.
\end{multline} 
Thus,~\eqref{eq: upp bd phi ij + 2} follows 
from~\eqref{eq: upp bd phi ij + use},~\eqref{eq: circ int 1 1} and~\eqref{eq: circ int 2}, applied with $\boldsymbol\sigma=+$.
\end{proof}


\begin{lemma}\label{lem: main counting 1}
There exists $\eta$ depending on $\eta'$ and some $\bar C_1$ so that 
\be\label{eq: final 1}
N_t(\Omega\setminus \ball)=\\
\bar C_1(\beta-\alpha)e^{2t}+\mathcal M_0+ O\Bigl((1+\absolute{\alpha}+\absolute{\beta})^N e^{(2-2\eta)t}\Bigr)
\ee
where $N$ is absolute, the implied constants depend on $R$ and 
\[
\mathcal M_0=\#\{v\in \Lambda'_{\rm nz}\cap (L\cup L')\cap e^t(\Omega\setminus \ball): \alpha\leq \sqf(v)\leq \beta\}.
\]

Similar assertion holds with $\Omega\setminus\ball$ replaced by $\ball\setminus e^{-1}\ball$.
\end{lemma}

\begin{proof}
We will prove the assertion for $\Omega\setminus \ball$, the proof for $\ball\setminus e^{-1}\ball$ is similar. 

Recall that $\Omega\setminus \ball=\Omega_1\cup \Omega_2$ where 
\begin{align*}
    \Omega_1&=\Bigl\{(v_1, v_2)\in\Omega: \norm{v_1}> 1\Bigl\}\quad\text{and}\\
    \Omega_2&=\Bigl\{(v_1, v_2)\in\Omega: \norm{v_1}\leq 1, \norm{v_2}>1\Bigl\}.
\end{align*}

Fix $k=1$ or $2$. By~\eqref{eq: upp and lower bd phi ij 2}, for all $(i,j)\in\mathring{\mathcal I}_k^{-}$,  
\begin{multline}\label{eq: upp and lower bd phi ij 3} 
e^{2t}\Bigl(\Upsilon_{i,j}^-+ O\Bigl(\Sob(f_{i,j}^{{\boldsymbol\sigma}})\Sob(\xi_i^+)e^{-\delta_2 t}\Bigr)\Bigr)\leq \\
(1+O(\vare))\cdot N_t'(\Omega_{i,j})\leq(1+O(\vare))\cdot N_t(\Omega_{i,j}^+),
\end{multline}
where we used $\Omega_{i,j}\subset\Omega_{i,j}^+$ in the second inequality,~\eqref{eq: def Omega ij 2}. 

Also by~\eqref{eq: upp bd phi ij + 2}, for all $\varphi_{i, j}^+\in\mathcal I_k^+$, we have  
\be\label{eq: upp bd phi ij + 3}
(1+O(\vare))\cdot N_t'(\Omega_{i,j}^+)\leq e^{2t}\Bigl(\Upsilon_{i,j}^++O\Bigl(\Sob(f_{i,j}^{+})\Sob(\xi_i^+)e^{-\delta_2 t}\Bigr)\Bigr).
\ee

Thus summing~\eqref{eq: upp and lower bd phi ij 3} over all 
$(i,j)\in\mathring{\mathcal I}_k^{-}$,  
\begin{multline}\label{eq: upp and lower bd phi ij 4} 
e^{2t}\sum_{\mathring{\mathcal I}_k^{-}}\Bigl(\Upsilon_{i,j}^-+ O\Bigl(\Sob(f_{i,j}^{+})\Sob(\xi_i^+)e^{-\delta_2 t}\Bigr)\Bigr)\leq\\
(1+O(\vare))\sum_{\mathring{\mathcal I}_k^{-}} N_t'(\Omega_{i,j})\leq(1+O(\vare))\sum_{\mathring{\mathcal I}_k^{-}} N_t'(\Omega_{i,j}^+)
\end{multline}
Moreover, summing~\eqref{eq: upp bd phi ij + 3} over all $(i,j)\in\mathcal I_k^+$, we get the following:  
\begin{multline}\label{eq: upp and lower bd phi ij 5} 
(1+O(\vare))\sum_{\mathring{\mathcal I}_k^{-}} N_t'(\Omega_{i,j}^+)\leq 
(1+O(\vare))\sum_{\mathcal I_k^+} N_t'(\Omega_{i,j}^+)\leq \\
e^{2t}\sum_{\mathcal I_k^+}\Bigl(\Upsilon_{i,j}^++O\Bigl(\Sob(f_{i,j}^{+})\Sob(\xi_i^+)e^{-\delta_2 t}\Bigr). 
\end{multline}

By~\eqref{eq: Omega ij 1} and~\eqref{eq: Omega ij 2}, $\Omega_{i,j}\subset \Omega_k$ are disjoint and 
$\Omega_k\subset \bigcup_{\mathcal I_k^+}\Omega_{i,j}^+$. Hence,~\eqref{eq: upp and lower bd phi ij 4} implies that 
\be\label{eq: upp and lower bd phi ij 6}
(I)\leq (1+O(\vare))N_t'(\Omega_k)\leq (II).
\ee
where $(I)$ is the first line in~\eqref{eq: upp and lower bd phi ij 4} and $(II)$ is the last line in~\eqref{eq: upp and lower bd phi ij 5}. 

Recall from Lemma~\ref{lem: upp and lower bd 2} that 
\[
\Upsilon_{i,j}^\pm=c_\Lambda (\beta-\alpha)\absolute{I_{i,j, \pm}}\int_\R \xi_i^\pm\int_\R\varrho_j^\pm
\] 
in view of~\ref{xi-1},~\ref{rho-1}, and~\eqref{eq: def I ij -}, the above implies that  
\[
\sum_{\mathcal I_k^+}\Upsilon_{i,j}^+= (1+O(\vare))\sum_{\mathring{\mathcal I}_k^{-}}\Upsilon_{i,j}^-=(1+O(\vare))(\beta-\alpha) \bar C_{k,1}
\]
where $\bar C_{k,1}$ is absolute and the implied constants depend on $R$. 

Furthermore, using $\vare=e^{-\eta'}$, we conclude  
\begin{align*}
\sum_{i,j}\Sob(f_{i,j,}^\pm)\Sob(\xi_i^\pm)e^{-\delta_2 t}&\ll (1+\absolute{\alpha}+\absolute{\beta})^N\vare^{-N}e^{-\delta_2 t}\\
&\ll (1+\absolute{\alpha}+\absolute{\beta})^N e^{-\delta_2t/2 },
\end{align*}
where the implied constant depends on $R$ and we assume $\eta'$ is small enough so that $\delta_2-N\eta'>\delta_2/2$. 

Altogether, there is some $\eta>0$ so that for $k=1,2$, we have  
\[
N_t'(\Omega_k)=\bar C_{k,1}(\beta-\alpha)e^{2t} + (1+\absolute{\alpha}+\absolute{\beta})^N e^{(2-2\eta)t}
\]
Since $\Omega\setminus \ball=\Omega_1\cup\Omega_2$ is a disjoint union, we conclude that
\be\label{eq: upp and lower bd phi ij 7}
N_t'(\Omega\setminus \ball)=\bar C_{1}(\beta-\alpha)e^{2t} + (1+\absolute{\alpha}+\absolute{\beta})^N e^{(2-2\eta)t}
\ee
where $\bar C_{1}=\bar C_{1,1}+\bar C_{2,1}$. 

The lemma follows from~\eqref{eq: upp and lower bd phi ij 7} and the definition of $\mathcal M_0$.  
 \end{proof}


\subsection*{Proof of Theorem~\ref{thm: main}}
We will again use the following 
\be\label{eq: contribution of vectors < et}
\#\{v\in \Lambda'\cap e^{\frac{2t}5}\ball\}\leq C_1'e^{\frac{8t}5}
\ee
where $C_1'$ depends on $R$, see~\eqref{eq: contribution of vectors < et 1}.

First Apply Lemma~\ref{lem: main counting 1}, with $t$ and $\Omega\setminus\ball$. Then 
\begin{multline}\label{eq: final ell 0}
N_t'(\Omega\setminus\ball)=\\
{\bar{C_1}}(\beta-\alpha)e^{2(t-\ell)}+\mathcal M'+ O\Bigl((1+\absolute{\alpha}+\absolute{\beta})^N e^{(2-2\eta)(t-\ell)}\Bigr)
\end{multline}
where
\[
\mathcal M'=\#\Bigl\{v\in \Lambda'_{\rm nz}\cap (L\cup L')\cap e^{t}\Bigl(\Omega\setminus \ball\Bigr): \alpha\leq \sqf(v)\leq \beta\Bigr\}
\]

We now control the contribution of $\Lambda'\cap e^t\ball$ to the count.
Recall our notation $\ball(e^{-\ell})=e^{-\ell}\ball$. Then $e^t\ball(e^{-\ell})=e^{t-\ell}\ball$, and 
\[
e^{t-\ell}\Bigl(\ball\setminus e^{-1}\ball\Bigr)=e^t(\ball(e^{-\ell})\setminus (e^{-1}\ball(e^{-\ell})).
\]
Applying Lemma~\ref{lem: main counting 1}
with $t-\ell$ (instead of $t$) for $\ell\leq 3t/5$ and $\ball\setminus e^{-1}\ball$, 
\begin{multline}\label{eq: final ell}
N_t'(\ball(e^{-\ell})\setminus e^{-1}\ball(e^{-\ell}))=\\
\bar{\bar{C_1}}(\beta-\alpha)e^{2(t-\ell)}+\mathcal M_\ell+ O\Bigl((1+\absolute{\alpha}+\absolute{\beta})^N e^{(2-2\eta)(t-\ell)}\Bigr)
\end{multline}
where
\[
\mathcal M_\ell=\#\Bigl\{v\in \Lambda'_{\rm nz}\cap (L\cup L')\cap e^{t}\Bigl(\ball(e^{-\ell})\setminus e^{-1}\ball(e^{-\ell})\Bigr): \alpha\leq \sqf(v)\leq \beta\Bigr\}
\]
and $L,L'$ are as in Lemma~\ref{lem: exceptional subspaces t-ell}.

Summing~\eqref{eq: final ell} over $0\leq \ell\leq 3t/5$, we get 
\[
N_t(\ball\setminus e^{-3t/5}\ball)=\bar{\bar{C_1}}(\beta-\alpha)e^{2t}+\mathcal M''+ O((1+\absolute{\alpha}+\absolute{\beta})^N e^{(2-\eta)t})
\]
where $\mathcal M''=\sum\mathcal M_\ell$. This,~\eqref{eq: final ell 0} and~\eqref{eq: contribution of vectors < et} thus imply
\be\label{eq: final 3}
N_t(\Omega)=\\
C_1(\beta-\alpha)e^{2t}+\mathcal M+ O\Bigl((1+\absolute{\alpha}+\absolute{\beta})^N e^{(2-\eta)t}\Bigr)
\ee
where $\mathcal M=\#\{v\in \Lambda'_{\rm nz}\cap (L\cup L')\cap e^t\Omega: \alpha\leq \sqf(v)\leq \beta\}$.

\medskip

To conclude the proof, we rewrite~\eqref{eq: final 3} in the notation of Theorem~\ref{thm: main} 
and further analyze $\mathcal M$. Recall that $t=\frac12\log T$, hence, by~\eqref{eq: counting Omega'} and~\eqref{eq: final 3},
\be\label{eq: final 4}
R_\torus(\alpha,\beta, T)=
C_1(\beta-\alpha)T+\mathcal M+ O\Bigl((1+\absolute{\alpha}+\absolute{\beta})^N T^{1-\frac{\eta}{2}}\Bigr). 
\ee
We now turn to the term $\mathcal M$. Since 
\[
\sqf(g_1w_1, g_2w_2)=\sqf(g_2^{-1}g_2w_1,w_2)\quad\text{and}\quad g_2^{-1}g_2=\begin{pmatrix}\mathsf a & \mathsf b\\ \mathsf b &\mathsf c\end{pmatrix}.
\]
We conclude, as in the proof of Lemma~\ref{lem: super exceptional}, that  
if we put $w_1=(\cox_1,\coy_1)$ and $w_2=(-\coy_2,\cox_2)$, then $u_i=(\cox_i,\coy_i)$ satisfy 
\be\label{eq: u1 u2 final}
\begin{aligned}
&\norm{u_i}\leq \max\{\norm{g_1^{\pm1}},\norm{g_2^{\pm1}}\}e^{\delta_1t/\bar A}\leq e^{2\delta_1t/\bar A}\quad\text{and} \\
&\absolute{B_\torus(u_1,u_2)}=\absolute{\sqf(g_2^{-1}g_1w_1,w_2)}\leq e^{-2\delta_1t/5}.
\end{aligned}
\ee 
where we assumed $t$ is large in the second inequality of the first line. 
Thus by Lemma~\ref{lem: super exceptional}, the pair $(w_1',w_2')$ is obtained from $(u_2, u_1)$ using the above relation, that is, 
$w'_1=(\cox_2,\coy_2)$ and $w'_2=(-\coy_1,\cox_1)$. 

Let $v\in \Lambda'\cap (L\cup L')\cap e^t\Omega$ satisfy that $\alpha\leq\sqf(v)\leq \beta$.
For simplicity, let us assume that $v\in L$ and write $v=\ell_1(g_1w_1,0)+\ell_2(0,g_2w_2)$. Then,
\[
v=(v_1+v_2, \omega(v_1-v_2))=(\ell_1g_1w_1, \ell_2g_2w_2)
\] 
where $v_i\in2\pi\Delta^*$ and $\norm{v_i}\leq e^t$. Recall also that $(g_1,g_2)=(g_\torus, -\omega g_\torus\omega)$ and $g_\torus\Z^2= 2\pi\Delta^*$, hence,  
\[
\begin{aligned}
v_1&=g_{\torus}\tfrac{\ell_1w_1-\ell_2\omega w_2}{2}=g_{\torus}\tfrac{\ell_1u_1+\ell_2u_2}{2}\\
v_2&=g_{\torus}\tfrac{\ell_1w_1+\ell_2\omega w_2}{2}=g_{\torus}\tfrac{\ell_1u_1-\ell_2u_2}{2};
\end{aligned}
\]
changing $L$ to $L'$ yields $v_1=g_{\torus}\frac{\ell_1u_1+\ell_2u_2}{2}$ and $v_2=g_{\torus}\tfrac{-\ell_1u_1+\ell_2u_2}{2}$.  

Altogether,~\eqref{eq: B and Q0'} implies that  
\[
\mathcal M=\#\left\{(\ell_1,\ell_2): \begin{array}{c}g_\torus\tfrac{\ell_1u_1+\ell_2u_2}{2}=v_1, g_\torus\tfrac{\ell_1u_1-\ell_2u_2}{2}=v_2\\
v_i\in2\pi\Delta^*, \norm{v_i}\leq e^t, \alpha\leq\norm{v_1}^2-\norm{v_2}^2\leq \beta\end{array}\right\}.
\]
By Lemma~\ref{lem: special subspace 2}, applied with $2\delta_1/\bar A$ and $\bar A/5$, we conclude that 
\[
\mathcal M\ll \max(\absolute{\alpha},\absolute{\beta}) e^{(2-\frac{2\delta_1}{\bar A}) t}=\max(\absolute{\alpha},\absolute{\beta})T^{1-\frac{\delta_1}{\bar A}},
\] 
where the implied constant depends on $\mathsf a$, $\mathsf b$, and $\mathsf c$ unless 
\[
\absolute{B_\torus(u_1,u_2)}\leq e^{(-2+\frac{2\delta_1}{\bar A})t}\leq T^{-1+\delta}.
\] 

Let $\kappa=\min\{\eta/2,\delta_1/\bar A\}$. Altogether, we conclude that 
\[
R_\torus(\alpha,\beta, T)=
C_1(\beta-\alpha)T + O\Bigl((1+\absolute{\alpha}+\absolute{\beta})^N T^{1-\kappa}\Bigr)
\]
unless $\{u_1,u_2\}$ satisfy~\eqref{eq: exceptional thm main}, in which case, we have 
\be\label{eq: final 5}
R_\torus(\alpha,\beta, T)=
C_1(\beta-\alpha)T+\mathcal M+ O\Bigl((1+\absolute{\alpha}+\absolute{\beta})^N T^{1-\kappa}\Bigr). 
\ee
We now show that $\mathcal M=M_T(u_1,u_2)$. Let $(\ell_1,\ell_2)$ be as in the definition of $\mathcal M$, then  
\[
B_\torus\bigl(\tfrac{\ell_1u_1+\ell_2u_2}{2}\bigr)=\norm{g_\torus\tfrac{\ell_1u_1+\ell_2u_2}{2}}^2=\norm{v_1}^2\leq e^{2t}=T
\]
Similarly for $v_2=\tfrac{\ell_1u_1-\ell_2u_2}{2}$. Moreover, we have 
\begin{align*}
\norm{v_1}^2-\norm{v_2}^2&=B_\torus\bigl(\tfrac{\ell_1u_1+\ell_2u_2}{2}\bigr)- B_\torus\bigl(\tfrac{\ell_1u_1-\ell_2u_2}{2}\bigr)\\
&=B_\torus(u_1,u_2)\ell_1\ell_2\in[\alpha,\beta].
\end{align*}
Thus $(\ell_1/2,\ell_2/2)$ satisfies the conditions in the definition $M_T(u_1,u_2)$. Similarly if $(\ell_1',\ell_2')$ satisfies the conditions in the definition $M_T(u_1,u_2)$, then $(2\ell_1',2\ell_2')$ satisfy the conditions in the definition of $\mathcal M$. 

The proof is complete.   
\qed

\subsection*{Proof of Corollary~\ref{cor: Stronger Diophantine}}
We first prove part~(1). Recall our assumption that there exist $A,q>0$ so that for all $(m,n,k)\in\Z^3$ we have 
\be\label{eq: linear form Dio proof}
\absolute{\mathsf a m+\mathsf b n+\mathsf c k}> q\norm{(m,n,k)}^{-A}.
\ee
This implies that~\eqref{eq: Diophantine condition} holds for some $A'$, depending on $A$, and all $T\geq T_0(A, q)$. 
Furthermore, in view of~\eqref{eq: linear form Dio proof}, for $u_i=(\cox_i,\coy_i)\in\Z^2$, we have 
\begin{multline*}
\absolute{B_\torus(u_1,u_2)}= \absolute{\mathsf a \cox_1\cox_2 +\mathsf b(\coy_1\cox_2+\cox_1\coy_2)+\mathsf c\coy_1\coy_2}>\\
q\norm{(\cox_1\cox_2, \coy_1\cox_2+\cox_1\coy_2,\coy_1\coy_2)}^{-A},
\end{multline*}
which implies~\eqref{eq: exceptional thm main} does not hold so long as $\delta$ is small enough. In view of Theorem~\ref{thm: main}, this finishes the proof of part~(1).

The proof of part~(2) is similar. Recall that $\mathsf b=0$ and $\mathsf{ac}=1$. 
By our assumption there exist $A,q>0$ so that for all $(m,n)\in\Z^2$, we have 
\be\label{eq: linear form Dio proof part 2}
\absolute{\mathsf a^2 m+n}> q\norm{(m,n)}^{-A}.
\ee
As in the previous case, we conclude that ~\eqref{eq: Diophantine condition} holds for some $A'$, depending on $A$, and all $T\geq T_0(A, q)$. Hence, by Theorem~\ref{thm: main}, either 
\[
\absolute{R_{\torus}(\alpha,\beta,T)-\pi^2(\beta-\alpha)} \leq  C (1+\absolute{\alpha}+\absolute{\beta})^{N}T^{-\kappa},
\] 
which implies the claim in this part, or there are $u_1,u_2 \in \Z^2\setminus\{0\}$ so that 
\be\label{eq: exceptional part 2 proof}
\norm{u_1},\norm{u_2} \leq T^{\delta/A} \qquad\text{and}  \qquad\absolute{B_\torus (u_1,u_2)}\leq T^{-1+\delta}
\ee
and moreover 
\be\label{eq: R torus proof of 2}
R_{\torus}(\alpha,\beta,T)-\pi^2(\beta-\alpha)=\frac{M_T(u_1,u_2)}{T} + O\Bigl( C (1+\absolute{\alpha}+\absolute{\beta})^{N}T^{-\kappa}\Bigr)
\ee
where 
\[
M_T(u_1,u_2) = \#\left\{(\ell_1,\ell_2)\in \tfrac12\Z^2: \begin{array}{l}\ell_1u_1\pm \ell_2u_2\in\Z^2,\\ \ B_\torus (\ell_1u_1\pm \ell_2u_2)\leq T,\\
\ \ {4B_\torus(u_1,u_2)}\ell_1\ell_2 \in {[\alpha,\beta]}\end{array}\right\}.
\]

By Lemma~\ref{lem: super exceptional}, if $T_0$ is large enough, 
then $B_\torus(u_1, u_2)=0$. Hence $M_T(u_1, u_2)$ does not contribute to $R'_\torus(\alpha, \beta)$. This and~\eqref{eq: R torus proof of 2} finish the proof of this case and of the corollary.    
\qed

\section{Equidistribution of expanding circles}\label{sec: equidistribution} 
In this section we prove an effective equidistribution result for circular averages; the proof is based on~\cite{LMW22}.

Let $G=\SL_2(\R)\times\SL_2(\R)$ and let $\Gamma\subset G$ be a lattice; put $X=G/\Gamma$. Let $m_X$ denote the $G$-invariant probability measure on $X$. 

We fix a right invariant metric on $G$ using the Killing form and the maximal compact subgroup $\SO(2)\times\SO(2)$, and let $d_X$ denote the induced metric on $X$. There exists $D'$ so that for all $\tau\geq 2$ and all $\theta\in\R$, 
\be\label{eq: d-X Lipschtiz}
d_X(x,x')\leq e^{D'\tau} d_X(\Delta(a_\tau r_\theta)x, \Delta(a_\tau r_\theta)x') 
\ee

For the convenience of the reader, we give again the statement of Theorem~\ref{thm: r effective equid}:

{
\renewcommand{\thesubsection}{\ref{thm: r effective equid}}
\begin{thm}\label{thm: r effective equid'}
Assume $\Gamma$ is arithmetic. 
For every $x_0\in X$, and large enough $R$ (depending explicitly on $X$ and the injectivity radius at $x_0$), for any $e^t\geq R^{D}$, at least one of the following holds.
\begin{enumerate}
\item For every $\varphi\in C_c^\infty(X)$ and $2\pi$-periodic smooth function $\xi$ on $\R$, we have 
\[
\biggl|\ave \varphi(\Delta(a_tr_\theta) x_0)\xi(\theta)\diff\!\theta-\int_0^{2\pi}\xi(\theta)\diff\!\theta \int \varphi\diff\!m_X\biggr|\leq \Sob(\varphi)\Sob(\xi)R^{-\kappa_0}
\]
where we use $\Sob(\cdot)$ to denote an appropriate Sobolev norm on both $X$ and $\R$ respectively. 
\item There exists $x\in X$ such that $Hx$ is periodic with $\vol(Hx)\leq R$, and 
\[
\dist_X(x,x_0)\leq R^{D}t^De^{-t}.
\] 
\end{enumerate} 
The constants $D$ and $\kappa_0$ are positive and depend on $X$ but not on $x_0$. 
\end{thm}
}
\addtocounter{subsection}{-1}

\begin{proof}
Fix $0<{\zeta}_0<1/10$ such that the $U^-AU$ decomposition is an analytic diffeomorphism on the identity neighborhood of radius $2\zeta_0$ in $\SL_2(\R)$ where $U^-$ is the subgroup of lower triangular unipotent matrices, $U$ is the subgroup of upper triangular unipotent matrices, and $A$ is the subgroup of diagonal matrices. In particular, there are analytic diffeomorphism $s^-$, $\tau$, $s$ from $(-\zeta_0,\zeta_0)$ to neighborhoods of $0$ in $(-1,1)$, such that $r_\zeta=u^-_{s^-(\zeta)}a_{\tau(\zeta)}u_{s(\zeta)}$. Note that 
\be\label{eq: tau s s-}
\tau(\zeta)=O(\zeta^2),\  s(\zeta)=\zeta+O(\zeta^2),\ s^{-}(\zeta)=-\zeta+O(\zeta^2),
\ee
and $\frac{\diff}{\diff\!\zeta}s=1+O(\zeta)$.

Using this we approximate the circular average (on small intervals) with unipotent average. First note that
\[
\begin{aligned}
\Delta(a_tr_{\hat\zeta+\zeta})x_0&=\Delta(a_tu^-_{s^-(\zeta)}a_{\tau(\zeta)}u_{s(\zeta)}r_{\hat\zeta})x_0\\
&=\Delta(a_tu^-_{s^-(\zeta)}a_{-t}a_{\tau(\zeta)})\Delta(a_tu_{s(\zeta)}r_{\hat\zeta})x_0
\end{aligned}
\]
is within distance $O(e^{-2t}s^-(\zeta)+\tau(\zeta))=O(e^{-2t}\zeta+\zeta^2)$ from $\Delta(a_tu_{s(\zeta)}r_{\hat\zeta})x_0$. Therefore for all $0\leq \zeta\leq\zeta_0$ we have 
\begin{multline*}
\frac1\zeta\int_0^\zeta\phi(\Delta(a_tr_{\hat\zeta+\theta})x_0)\diff\!\theta=\\
\frac1\zeta\int_0^\zeta\phi(\Delta(a_tu_{s(\theta)}r_{\hat\zeta})x_0)\diff\!\theta+O\Big(\Sob(\phi)(e^{-2t}\zeta+\zeta^2)\Big)=\\
\frac1\zeta\int_0^{s(\zeta)}\phi(\Delta(a_tu_\theta r_{\hat\zeta}(x_0)(s^{-1}(\theta))'\diff\!\theta+O\Big(\Sob(\phi)(e^{-2t}\zeta+\zeta^2)\Big)
\end{multline*}
where we used the above estimate in the first equality and a change of variable in the second equality. 

Since $s(\zeta)-\zeta=O(\zeta^2)$, see~\eqref{eq: tau s s-}, we conclude that 
\[
\frac1\zeta\int_0^\zeta\phi(\Delta(a_tr_{\hat\zeta+\theta})x_0)\diff\!\theta=
\frac1\zeta\int_0^{\zeta}\phi(\Delta(a_tu_\theta r_{\hat\zeta})x_0)(s^{-1}(\theta))'\diff \theta+O\Big(\Sob(\phi)\zeta\Big)
\]
where we used $e^{-2t}\zeta+\zeta^2\leq 2\zeta$. 

Similarly, using $\sup_{\theta\in(0,\zeta)}\absolute{(s^{-1}(\theta))'-1}\ll \zeta$ and a change of variable,   
\be\label{eq: circular 1}
\begin{aligned}
\frac1\zeta\int_0^\zeta\phi(\Delta(a_tr_{\hat\zeta+\theta})x_0)\diff\!\theta&=\frac1\zeta\int_0^{\zeta}\phi(\Delta(a_tu_\theta r_{\hat\zeta})x_0)\diff\!\theta+O\Big(\Sob(\phi)\zeta\Big)\\
&=\int_0^1\phi(\Delta(a_tu_{\zeta s} r_{\hat\zeta})x_0)\diff\!s+O\Big(\Sob(\phi)\zeta\Big). 
\end{aligned}
\ee
Let $\tau=-(\log\zeta)/2$. Then 
\be\label{eq: circular 2}
\begin{aligned}
 \int_0^1\phi(\Delta(a_tu_{\zeta s} r_{\hat\zeta})x_0)\diff\!s&=\int_0^1\phi(\Delta(a_{t-\tau}a_{\tau}u_{\zeta s}a_{-\tau}a_{\tau}r_{\hat\zeta})x_0)\diff\!s\\
 &=\int_0^1\phi(\Delta(a_{t-\tau}u_sa_{\tau}r_{\hat\zeta})x_0)\diff\!s.
\end{aligned}
\ee

Let $D_1$ and $\kappa_1$ be the constants given by~\cite[Thm.\ 1.1]{LMW22} applied with $X$ ($D_1$ denotes $A$ in~\cite[Thm.\ 1.1]{LMW22}). We will show the proposition holds with
\[
D=D_1+D'+1
\]
where $D'$ is as in~\eqref{eq: d-X Lipschtiz}. 

Let $T=e^{t-\tau}$ and $R=e^{D''\tau}$ for some $D''\geq 1$ which will explicated momentarily. Assume $e^t\geq R^D$, then 
\be\label{eq: T is large}
T=e^{t-\tau}=e^tR^{-1/D''}\geq R^{D-1}\geq R^{D_1}.
\ee

Apply~\cite[Thm.\ 1.1]{LMW22}, with $x_{\hat\zeta}:=\Delta(a_{\tau}r_{\hat\zeta})x_0$, $T\geq R^{D_1}$, see~\eqref{eq: T is large}, then so long as $D''$ is large enough, at least one of the following holds:

\medskip
\noindent
{\bf Case 1:} 
For every $\hat\zeta\in[0,2\pi]$ and all $\phi\in C_c^\infty(X)$, 
\be\label{eq: LMW applied}
\left|\int_0^1\phi(\Delta(a_{\log T}u_s)x_{\hat\zeta})\diff\!s-\int\phi\diff m_X\right|\leq \Sob(\phi)R^{-\kappa_1}.
\ee

\noindent
{\bf Case 2:} For some $\hat \zeta\in[0,2\pi]$, there exists $x\in X$ such that $Hx$ is periodic with $\vol(Hx)\leq R$ and 
\be\label{eq: case 2}
d_X(x,x_{\hat\zeta})\leq R^{D_1}(\log T)^{D_1}T^{-1}.
\ee

We will show that part~1 in the proposition holds if case 1 holds and part~2 in the proposition hols if case 2 holds.

\medskip

Let us first assume that case~1 holds. We begin with the following computation.  
\begin{multline}\label{eq: circular 3}
\int_0^{2\pi}\phi(\Delta(a_tr_\theta )x_0)\xi(\theta)\diff\!\theta=\\
\frac1{ \zeta}\int_{\hat\zeta=0}^{2\pi}\int_0^\zeta\phi(\Delta(a_tr_{\hat\zeta+\theta})x_0)\xi(\hat\zeta+\theta)\diff\!\theta\diff\!\hat\zeta=\\
\frac1{\zeta}\int_{\hat\zeta=0}^{2\pi}\biggl(\int_0^\zeta\phi(\Delta(a_tr_{\hat\zeta+\theta})x_0)\diff\!\theta\biggr)\xi(\hat\zeta)\diff\!\hat\zeta+O(\Sob(\phi)\Sob(\xi)\zeta).
\end{multline}
Furthermore, by~\eqref{eq: circular 1} and~\eqref{eq: circular 2}, we have
\be\label{eq: circular 4}
\frac1\zeta\int_0^\zeta\phi(\Delta(a_tr_{\hat\zeta+\theta})x_0)\diff\!\theta=\int_0^1\phi(\Delta(a_{\log T}u_s)x_{\hat\zeta})\diff\!s+O(\Sob(\phi)\zeta). 
\ee
Altogether, using~\eqref{eq: LMW applied},~\eqref{eq: circular 3}, and~\eqref{eq: circular 4}, we conclude that 
\begin{multline}\label{eq: r effective equid 1}
\biggl|\ave \varphi(\Delta(a_t\,r_\theta) x_0)\xi(\theta)\diff\!\theta-\int_0^{2\pi}\xi(\theta)\diff\!\theta \int \varphi\diff\!m_X\biggr|\leq \\\Sob(\varphi)\Sob(\xi)R^{-\kappa_0};
\end{multline}
where $\kappa_0=\min\{\kappa_1, 2/D''\}$ --- we used $\zeta=e^{2\tau}= R^{2/D''}$.  
Thus, part~1 in the proposition holds if case~1 holds.

\medskip

Let us now assume that case~2 holds and let $x_{\hat\zeta}=\Delta(a_{\tau}r_{\hat\zeta})x_0$ be as in~\eqref{eq: case 2}. Then by~\eqref{eq: d-X Lipschtiz}, we have
\begin{align*}
d_X(\Delta(a_{\tau}r_{\hat\zeta})^{-1}x, x_0)&\leq e^{D'\tau} R^{D_1}(\log T)^{D_1}T^{-1}\\
&\leq e^{(1+D')\tau} R^{D_1} t^{D_1} e^{-t}\leq R^Dt^De^{-t}.
\end{align*}
Furthermore, $\Delta(a_{\tau}r_{\hat\zeta})^{-1}x$ has a periodic $H$-orbit of volume $\leq R$. Thus part~2 in the proposition holds in this case.
The proof is complete.
\end{proof}


\section{Cusp functions of Margulis and the upper bound}\label{eq: upper bound}
In this section, we put 
\[
\Gamma=\SL_2(\Z)\times\SL_2(\Z)\subset G.
\]
Recall the following definition. 
{
\renewcommand{\thesubsection}{\ref{def: special}}
\begin{definition}
Let $g=(g_1,g_2)\in G$.
A two dimensional $g\Z^4$-rational linear subspace $L\subset\R ^4$ is called $(\rho, A,t)$-exceptional if 
there are $(v_1,0), (0,v_2)\in\Z^4$ satisfying  
\be\label{eq: quasi-null}
\norm{g_1v_1}, \norm{g_2v_2}\leq e^{\rho t} \quad\text{and}\quad \absolute{\sqf(g_1v_1, g_2v_2)}\leq e^{-A\rho t}
\ee
so that $L$ is spanned by $\{(g_1v_1,0), (0,g_2v_2)\}$. 

Given a $(\rho, A,t)$-special subspace $L$, we will refer to $\{(g_1v_1,0), (0,g_2v_2)\}$ as a {\em spanning set} for $L$.   
\end{definition}
}
\addtocounter{subsection}{-1}

Let $f_i\in C_c(\R^2)$, and define $f$ on $\R^4$ by $f(w_1,w_2)=f_1(w_1)f_2(w_2)$.  
For every $h\in \SL_2(\R)$, let
\be\label{eq: def tilde f}
\tilde f_{\rho,A, t}(h;g\Gamma)=\sum_{v\in {\Nt}(g\Z^4)} f(\Delta(h)v).
\ee
where ${\Nt}\Bigl(g\Z^4\Bigr)$ denotes the set of 
vectors in $g\Z^4$ not contained in any $(\rho,A, t)$-special subspace $L$ and also not contained in $\R^2\times\{0\}\cup\{0\}\cup\R^2$. 
In the sequel, we will often drop the dependence on $A$, $\rho$, and $t$ from the notation and denote 
$\tilde f_{\rho,A, t}(h;g\Gamma)$ by $\tilde f(h;g\Gamma)$.

\medskip

The following is one of the main result of this section.

\begin{propos}\label{prop: main}
For all $A_1\geq 10^3$ we have the following: Let $(g_1,g_2)\in G$.  
Then for all small enough $\rho$ and all large enough $t$ at least one of the following holds:
\begin{enumerate} 
\item Let $\mathcal C_t=\{\theta\in[0,2\pi]: \tilde f(a_t\rot_\theta; g\Gamma)\geq e^{A_1\rho t}\}$. Then 
\[
\int_{\mathcal C_t}\tilde f(a_t\rot_\theta;g\Gamma)\diff\!\theta\ll e^{-\rho^3 t/A_1}.
\]
where $\tilde f(h;g\Gamma)=\tilde f_{\rho, A_1, t}(h;g\Gamma)$, see~\eqref{eq: def tilde f}.
\item There exists $Q\in\Mat_2(\Z)$ whose entries are bounded by $e^{100\rho t}$ and $\lambda\in\R$ 
satisfying $\norm{g_2^{-1}g_1-{\lambda}Q}\ll e^{-(A_1-100)\rho t}$.
\end{enumerate}
The implied constants depend polynomially on $\norm{g_1}$ and $\norm{g_2}$. 
\end{propos}

The proof of this proposition occupies most of this section.

\subsection*{The cusp functions} 
Let $\pcal$ denote the set of primitive vectors in $\Z^2$. 
For any $h\in\SL_2(\R)$, define 
\be\label{eq: define omega-i}
\omega(h\SL_2(\Z))= \sup\Bigl\{1/\norm{hv}: v\in\pcal\Bigr\}. 
\ee

We begin with the following lemma.  

\begin{lemma}[cf.\ Lemma 7.4~\cite{EskinMasur-Upp}]
\label{lem: L2 bound SL2}
For every $0< p<2$, there exists $t_p$ and $b_p$ so that the following holds. 
For every $x\in \SL_2(\R)/\SL_2(\Z)$ and all $t\geq t_p$, we have  
\[
\ave\omega(a_t\rot_\theta x)^{p}\diff\!\theta\leq 2^{-t/t_p} \omega(x)^{p}+ b_p.
\] 
\end{lemma}

\begin{proof}
This is by now well known, see e.g.~\cite{EskinMozes-MF}. 
\end{proof}

\subsection*{The sets $\Theta_t(\delta)$ and $\Theta_t'(\delta)$}
To put an emphasis on the product structure of $G$ and $X$, 
we will often write $X=G_1/\Gamma_1\times G_2/\Gamma_2$ where $G_i=\SL_2(\R)$ and $\Gamma_i=\SL_2(\Z)$. 
Moreover, given $g=(g_1,g_2)\in G$, we write 
\be\label{eq: def omegai}
\omega_i(g_i\Gamma_i):=\omega(g_i\SL_2(\Z)). 
\ee

For $i=1,2$, let $x_i\in G_i/\Gamma_i$. For all $t\geq 0$ and every $0<\delta\leq 1/10$, let 
\be\label{eq: def Theta-t}
\Theta_t(\delta)\!=\!\Bigl\{\theta\in[0,2\pi]\!:\! \omega_2(a_t\rot_\theta x_2)^{1-2\delta}\leq \omega_1(a_t\rot_\theta x_1)\leq \omega_2(a_t\rot_\theta x_2)^{1+2\delta}\Bigr\}
\ee
and let $\Theta_t'(\delta)=[0,2\pi]\setminus\Theta_t(\delta)$. 

\medskip

We have the following 

\begin{lemma}\label{lem: alpha1 and alpha3}
Let $0<\delta<1/10$, and put 
\[
p_1=(2-2\delta)(1+\tfrac12\delta) \quad\text{and}\quad p_2=\frac{(2+2\delta)(1+\frac12\delta)}{1+2\delta};
\]
note that $p_1,p_2<2$. Let $t(\delta)=\max(t_{p_1}, t_{p_2})$ and $b(\delta)=\max(b_{p_1},b_{p_2})$
where the notation is as in Lemma~\ref{lem: L2 bound SL2}. Then for all $(x_1,x_2)\in X$ and all $t\geq t(\delta)$ 
\[
\int_{\Theta'_t(\delta)} \Bigl(\omega_1(a_t\rot_\theta x_1)\omega_2(a_t\rot_\theta x_2)\Bigr)^{1+\frac{1}{2}\delta}\diff\!\theta
\leq 2^{-t/t(\delta)}\Bigl(\omega_1(x_1)+\omega_2(x_2)\Bigr)+ 2b(\delta).
\]
\end{lemma}

\begin{proof}
Let us write $\Theta'_t(\delta)=\Theta'_{t,1}(\delta)\cup\Theta'_{t,2}(\delta)$, where 
\begin{align*}
\Theta'_{t,1}(\delta)&=\{\theta\in[0,2\pi]:\omega_2(a_t\rot_\theta x_1)< \omega_1(a_t\rot_\theta x_2)^{1-2\delta}\}\\
\Theta'_{t,2}(\delta)&=\{\theta\in[0,2\pi]:\omega_2(a_t\rot_\theta x_1)> \omega_1(a_t\rot_\theta x_2)^{1+2\delta}\}.  
\end{align*}

Using Lemma~\ref{lem: L2 bound SL2}, for every $t>t_{p_1}$ we have  
\begin{align*}
\int_{\Theta_{t,1}'(\delta)} \Bigl(\omega_1(a_t\rot_\theta x_1)\omega_2(a_t\rot_\theta x_2)\Bigr)^{1+\frac12\delta}\diff\!\theta&\leq \ave \omega_1(a_t\rot_\theta x_1)^{p_1}\diff\!\theta\\
&\leq 2^{-t/t_{p_1}}\omega_2(x_2)+b_{p_1}.
\end{align*}
Similarly, for every $t>t_{p_2}$, we have
\begin{align*}
\int_{\Theta_{t,2}'(\delta)} \Bigl(\omega_1(a_t\rot_\theta x_1)\omega_2(a_t\rot_\theta x_2)\Bigr)^{1+\frac12\delta}\diff\!\theta&\leq \ave \omega_2(a_t\rot_\theta x_2)^{p_2}\diff\!\theta\\
&\leq 2^{-t/t_{p_2}}\omega_1(x_1)+b_{p_2}.
\end{align*} 
The claim follows from these two estimates.  
\end{proof}

\subsection*{A Diophantine condition} 
The following lemma is a crucial input in the proof of Proposition~\ref{prop: main}. 

For every $t\geq 1$, let 
\begin{align*}
&\pcal_{t}=\{v\in\pcal: e^{t-1}\leq \norm{v}< e^t\}\\
&\pcal(t)=\{v\in\pcal: \norm{v}< e^t\}.
\end{align*}

\begin{lemma}\label{lem: not many solutions}
The following holds for all $A\geq 10^3$ and all $\rho\leq 1/(100A)$. 
Let $(g_1,g_2)\in G$, there exist 
$t_1\geq 1$, depending on $\rho$ and polynomially on $\norm{g_i}$, so that if 
$t\geq t_1$, then at least one of the following holds:
\begin{enumerate}
\item We have  
\[
\#\Bigl\{v_1\in\pcal_{t}: \exists v_2\in\pcal(t), \absolute{\sqf(g_1v_1,g_2v_2)}\leq e^{-A\rho t}\Bigr\}\ll e^{(2-\rho)t}
\]
where the implied constant depends polynomially on $\norm{g_i}$. 
\item There exist $Q\in\Mat_2(\Z)$ whose entries are bounded by $e^{100\rho t}$ and $\lambda\in\R$ 
satisfying $\norm{g_2^{-1}g_1-{\lambda}Q}\leq e^{-(A-100)\rho t}$.
\end{enumerate}
\end{lemma}

\begin{proof}
For simplicity in the notation, let us write $\eta=e^{-\rho t}$. Let $A\geq10^3$, and assume that   
\begin{multline}\label{eq: part 1 fails}
\#\Bigl\{v_1\in\pcal_{t}: \exists v_2\in\pcal(t), |\sqf(g_1v_1,g_2v_2)|\leq \eta^A\Bigr\}>\\
E(\norm{g_1}\norm{g_2})^E\eta e^{2t}.
\end{multline}
We will show that if $E$ is large enough, then part~(2) holds.
\medskip

Let us write  
\[
h:=g_2^{-1}g_1=\begin{pmatrix} a & b \\ c & d\end{pmatrix}.
\] 
Then~\eqref{eq: part 1 fails} and the fact that for any $q \in \SL(2,\R)$, \  $\Delta(q) \in\SO(\sqf)$ imply that if $t$ is large enough, 
depending on $\norm{h}$, for $\gg\eta e^{2t}$ 
many $v_1=(\cox_1,\coy_1)\in\pcal_t$ both of the following hold 
\begin{itemize}
\item We have $\absolute{c\cox_1+d\coy_1}\geq \eta^2 e^t$. 
\item There exists at least one $(\cox_2,\coy_2)\in\pcal(t)$ so that
\be\label{eq: well approximation}
\absolute{\sqf(h(\cox_1,\coy_1), (\cox_2,\coy_2))}\leq \eta^{A}.
\ee
\end{itemize}

Moreover, the fact that there are $\gg\eta e^{2t}$ vectors satisfying these two 
conditions implies that there are $v_1, v_1', v_1''\in\pcal_t$ satisfying the above two conditions so that  
\be\label{eq: v1, v1', v1''}
1\leq\absolute{\sqf(v,w)}\ll \eta^{-4}, \qquad \text{for $v, w\in\{v_1,v_1',v_1''\}$.}
\ee 

Let us fix three vectors $v_1,v_1',v_1''$ satisfying~\eqref{eq: v1, v1', v1''}, and let $v_2, v_2',v_2''$ 
be the corresponding vectors in $\pcal(t)$ satisfying~\eqref{eq: well approximation}, respectively. Then  
\be\label{eq: well approximation'}
hv_1=\mu v_2 + w_{1,2}
\ee
where $\mu\in\R$ satisfies $\absolute{\mu}\asymp 1$ and $\norm{w_{1,2}}\ll \eta^Ae^{-t}$ (recall that the implicit constants in these inequalities are allowed to depend polynomially on $\norm{h}$). Similarly, 
\[
hv'_1=\mu' v'_2 + w'_{1,2}\quad\text{and}\quad hv''_1=\mu'' v''_2 + w''_{1,2}
\]  
where $\mu',\mu''\in\R$ satisfy $\absolute{\mu'},\absolute{\mu''}\asymp 1$ and 
$\norm{w'_{1,2}}, \norm{w''_{1,2}}\ll \eta^Ae^{-t}$. 

With this notation we have  
\be\label{eq: well approximation''}
h(v_1\; v'_1)=(v_2\; v'_2)\begin{pmatrix}\mu & 0\\ 0 &\mu'\end{pmatrix}+ O(\eta^Ae^{-t})
\ee 
and similarly for $v_1, v''_1$ and $v'_1, v''_1$. Thus by \eqref{eq: v1, v1', v1''} 
\be\label{eq: v2', v2''}
1\leq\absolute{\sqf(v_2,v'_2)},\absolute{\sqf(v_2,v''_2)},\absolute{\sqf(v'_2,v''_2)}\ll \eta^{-4}.
\ee

In view of~\eqref{eq: v1, v1', v1''},~\eqref{eq: well approximation'},~\eqref{eq: well approximation''} and~\eqref{eq: v2', v2''} the conditions in Lemma~\ref{lem: Mobius} hold. The claim thus follows from Lemma~\ref{lem: Mobius} so long as $t$ is large enough to account for the constant $C$ in that lemma.
\end{proof}

\subsection*{Proof of Proposition~\ref{prop: main}}
Recall that $g=(g_1,g_2)$. Put 
\[
x_i=g_i\SL_2(\Z),\qquad \text{for $i=1,2$}.
\]
Let $A_1\geq 10^4$, $0<\rho<10^{-4}$ (small), and $t\geq 1$ (large) be so that Lemma~\ref{lem: not many solutions} holds for these choices. 
Put $\delta=2\rho^2/A_1$, and define $\Theta_t(\delta)$ and $\Theta_t'(\delta)$ as in~\eqref{eq: def Theta-t} with $t$ and $\delta$ and $x_i$. 
That is,
\[
\Theta_t(\delta)=\Bigl\{\theta\in[0,2\pi]: \omega_2(a_t\rot_\theta x_2)^{1-2\delta}\leq \omega_1(a_t\rot_\theta x_1)\leq \omega_2(a_t\rot_\theta x_2)^{1+2\delta}\Bigr\},
\]
and $\Theta'_t(\delta)=[0,2\pi]\setminus \Theta_t(\delta)$.
 
Apply Lemma~\ref{lem: not many solutions} with $A=A_1$ and $\rho$. 
If part~(2) in that lemma holds, then part~(2) in Proposition~\ref{prop: main} holds and the proof is complete. 
Thus, assume for the rest of the argument that part~(1) in Lemma~\ref{lem: not many solutions} holds. We will show that part~(1) in the Proposition~\ref{prop: main} holds.

Motivated by the definition of $\tilde f$ and Lemma~\ref{lem: special subspace}, define
\be\label{eq: def tilde omega}
\tilde\omega(a_t\rot_\theta;g\Gamma)=\sup\Bigl\{\Bigl(\|a_t\rot_\theta g_1v_1\|\|a_t\rot_\theta g_2v_2\|\Bigr)^{-1}: (v_1,v_2)\in\pcal^2(g)\Bigr\}
\ee
where $\pcal$ is the set of primitive vectors in $\Z^2$ and $\pcal^2(g)$ denotes the set of $(v_1,v_2)\in\pcal^2$ so that $\{(g_1v_1,0), (0,g_2v_2)\}$ is not a spanning set for any $(\rho, A_1, t)$-special subspace of $g\Z^4$, see Definition~\ref{def: special}.

It follows from the definition that  
\be\label{eq: tilde omega and prod of omega-i}
\tilde\omega\Bigl(a_t\rot_\theta; g\Gamma\Bigr)\leq \omega_1(a_t\rot_\theta x_1)\omega_2(a_t\rot_\theta x_2).
\ee
Put 
$\mathcal B_t=\{\theta\in[0,2\pi]:\tilde\omega(a_t\rot_\theta;g\Gamma)\!<\!\omega_1(a_t\rot_\theta x_1)\omega_2(a_t\rot_\theta x_2)\}$.

By a variant of Schmidt's Lemma, see also~\cite[Lemma 3.1]{EMM-Upp}, and the definition of $\tilde f$, we have     
\be\label{eq: Schmidt lemma}
\tilde f(a_t\rot_\theta;g\Gamma)\ll \tilde\omega\Bigl(a_t\rot_\theta ;g\Gamma\Bigr). 
\ee
Put $\tilde{\mathcal C}_t=\{\theta\in[0,2\pi]: \tilde\omega(a_t\rot_\theta;g\Gamma)\geq e^{A_1\rho t}\}$.
In view of~\eqref{eq: Schmidt lemma} and with this notation, it suffices to show that  
\be\label{eq: integral of tilde omega}
\int_{\tilde{\mathcal C}_t}\tilde\omega(a_t\rot_\theta;g\Gamma)\diff\!\theta\ll e^{-\rho^2 t/A_1}.
\ee

\subsection*{Contribution of $\mathcal B_t$}
Recall that if $\omega(h\SL_2(\Z))\geq 2$ for some $h\in \SL_2(\R)$, then there is some $v_h\in\pcal$ so that 
\be\label{eq: only one short vector}
\|hv_h\|^{-1}=\omega(h\SL_2(\Z))\quad\text{and}\quad \text{$\|hv\|>1/2$ for all $v_h\neq v\in\pcal$}.
\ee 

Let $\theta\in\mathcal B_t$. By the definition of $\tilde\omega$, there exist $v_1, v_2\in\pcal$ so that 
\[
\tilde\omega(a_t\rot_\theta; g\Gamma)=\norm{a_t\rot_\theta g_1v_1}^{-1}\norm{a_t\rot_\theta g_2v_2}^{-1}. 
\]
Since $\tilde\omega(a_t\rot_\theta; g\Gamma)<\omega_1(a_t\rot_\theta g_1\Gamma_1)\omega_2(a_t\rot_\theta g_2\Gamma_2)$, we conclude that 
\[
\min\Bigl\{\norm{a_t\rot_\theta g_1v_1}^{-1},\norm{a_t\rot_\theta g_2v_2}^{-1}\Bigr\}\leq 2.
\]
Therefore, for all such $\theta$, we have 
\[
\tilde\omega(a_t\rot_\theta; g\Gamma)\leq 2\max\{\omega_1(a_t\rot_\theta g_1\Gamma_1), \omega_2(a_t\rot_\theta g_2\Gamma_2)\}.
\]
Thus using Lemma~\ref{lem: L2 bound SL2}, we have 
\be\label{eq: contribution of Bt}
\begin{aligned}
\int_{\mathcal B_t\cap \tilde{\mathcal C}_t}&\tilde\omega(a_t\rot_\theta; g\Gamma)\diff\!\theta\leq e^{-\frac{A_1\rho t}2}\int_{\mathcal B_t}\tilde\omega(a_t\rot_\theta; g\Gamma)^{3/2}\diff\!\theta\\
&\leq 2 e^{-\frac{A_1\rho t}2}\int_0^{2\pi}\omega_1(a_t\rot_\theta x_1)^{\frac32}+\omega_1(a_t\rot_\theta x_2)^{\frac32}\diff\!\theta\ll e^{-\frac{A_1\rho t}2}.
\end{aligned}
\ee

Let $\Theta_t(\theta)$ and $\Theta'_t(\delta)$ be as above, and put 
\[
\tilde{\mathcal C}_t(\delta):=\tilde{\mathcal C}_t\cap \mathcal B_t^\complement\cap \Theta_t(\delta)\quad\text{ and }\quad\tilde{\mathcal C}_t'(\delta):=\tilde{\mathcal C}_t\cap \mathcal B_t^\complement\cap \Theta_t'(\delta).
\]

We consider the contribution of these two sets to $\int\tilde\omega$ separately --- indeed, controling the contribution of $\tilde{\mathcal C}_t(\delta)$ occupies bulk of the proof.  

\subsection*{Contribution of $\tilde{\mathcal C}_t'(\delta)$} 
By Lemma~\ref{lem: alpha1 and alpha3}, for all $t$ large enough, we have 
\[
\int_{\Theta'_t(\delta)}\Bigl(\omega_1(a_t\rot_\theta x_1)\omega_2(a_t\rot_\theta x_2)\Bigr)^{1+\frac12\delta}\diff\!\theta\ll 1
\]
From this and~\eqref{eq: tilde omega and prod of omega-i}, we conclude that 
\be\label{eq: contribution of C'}
\begin{aligned}
\int_{\tilde{\mathcal C}'_t(\delta)}&\tilde\omega(a_t\rot_\theta; g\Gamma)\diff\!\theta\leq \int_{\tilde{\mathcal C}'_t(\delta)}\omega_1(a_t\rot_\theta x_1)\omega_2(a_t\rot_\theta x_2)\diff\!\theta\\
&\leq e^{-\delta \rho A_1t/2} \int_{\Theta'_t(\delta)}\Bigl(\omega_1(a_t\rot_\theta x_1)\omega_2(a_t\rot_\theta x_2)\Bigr)^{1+\frac12\delta}\diff\!\theta\ll e^{-\rho^3 t}.
\end{aligned}
\ee

\subsection*{Contribution of $\tilde{\mathcal C}_t(\delta)$} 
Recall that
\[
\Theta_t(\delta)=\Bigl\{\theta\in[0,2\pi]: \omega_2(a_t\rot_\theta x_2)^{1-2\delta}\leq \omega_1(a_t\rot_\theta x_1)\leq \omega_2(a_t\rot_\theta x_2)^{1+2\delta}\Bigr\},
\]
and $\tilde{\mathcal C}_t(\delta)=\tilde{\mathcal C}_t\cap \mathcal B_t^\complement\cap\Theta_t(\delta)$. Note that the vectors which contribute to 
\be\label{eq: int tilde omega Ct}
\int_{\tilde{\mathcal C}_t(\delta)}\tilde\omega(a_t\rot_\theta ;g\Gamma)\diff\!\theta
\ee
satisfy $\Bigl\{(g_1v_1,g_2v_2): \norm{g_1v_1}, \norm{g_2v_2}\leq e^t\Bigr\}$.
It is more convenient to consider the cases $\norm{g_1v_1}\geq \norm{g_2v_2}$ and $\norm{g_1v_1}\leq \norm{g_2v_2}$ separately. 
As the arguments are similar in both cases, we assume $\norm{g_1v_1}\geq \norm{g_2v_2}$ for the rest of the proof.  

Recall our notation: for $t\geq 1$ 
\[
\pcal_t=\{v\in \pcal: e^{t-1}\leq \norm{v}< e^t\},
\]
and $\pcal(t)=\{v\in \pcal: \norm{v}\leq e^t\}$.

For every $n\in \N$ with $n\leq t+\log\|g_1\|+1=:t_1$, 
we investigate the contribution of $\pcal_n$ to~\eqref{eq: int tilde omega Ct}. 
For any $v_1\in \pcal_n$, let 
\[
I_{v_1}=\{\theta\in [0,2\pi]: \norm{a_t\rot_\theta g_1v_1}\leq 1/10\}.
\] 
Then the intervals $I_{v_1}$ are disjoint. 
Let $\tilde\pcal_n=\{v_1\in \pcal_n: I_{v_1}\cap\tilde{\mathcal C}_t(\delta)\neq \emptyset\}$.

Fix some $n\in\N$, $n\leq t_1$. Let $v_1\in\tilde\pcal_n$, and let $\theta\in I_{v_1}\cap\tilde{\mathcal C}_t(\delta)$. 
Then there exists $v_2\in\pcal$ so that 
\[
\tilde\omega(a_t\rot_\theta; g\Gamma)=\frac{1}{\norm{a_t\rot_\theta g_1v_1}\norm{a_t\rot_\theta g_2v_2}}. 
\]
Since $\theta\in\mathcal B_t$, we have 
$\tilde\omega(a_t\rot_\theta; g\Gamma)=\omega_1(a_t\rot_\theta x_1)\omega_2(a_t\rot_\theta x_2)$. 
Thus 
\be\label{eq: the is in Bt use}
\omega_i(a_t\rot_\theta x_i)=\norm{a_t\rot_\theta g_iv_i}^{-1}\quad\text{for $i=1,2$}. 
\ee

In view of~\eqref{eq: the is in Bt use}, and the definitions of $\mathcal B_t$ and $\Theta_t(\theta)$, thus  
\be\label{eq: L2 + epsilon : Theta-t}
\int_{\tilde{\mathcal C}_t(\delta)}\tilde\omega(a_t\rot_\theta ;g\Gamma)\diff\!\theta
\leq\sum_n\sum_{\tilde\pcal_n} \int_{I_{v_1}} \|a_t\rot_\theta g_1v_1\|^{-2-2\delta}. 
\ee

We also make some observations. 
Fix some $n\in\N$, $n\leq t_1$. Let $v_1\in\tilde\pcal_n$ and $\theta\in I_{v_1}\cap\tilde{\mathcal C}_t(\delta)$, and let $v_2\in\pcal$ be so that~\eqref{eq: the is in Bt use} holds. That is, 
$\omega_i(a_t\rot_\theta x_i)=\norm{a_t\rot_\theta g_iv_i}^{-1}$, for $i=1,2$,
and  
\[
\tilde\omega(a_t\rot_\theta; g\Gamma)=\Bigl(\norm{a_t\rot_\theta g_1v_1}\norm{a_t\rot_\theta g_2v_2}\Bigr)^{-1}.
\]
Since $\theta\in\tilde{\mathcal C_t}$, we have $\tilde\omega(a_t\rot_\theta;g\Gamma)\geq e^{A_1\rho t}$. This gives
\[
\norm{a_t\rot_\theta g_1v_1}\norm{a_t\rot_\theta g_2v_2}\leq e^{-A_1\rho t},
\] 
which implies that 
\[
\absolute{\sqf\Bigl(\darot(g_1v_1,g_2v_2)\Bigr)}=\absolute{\sqf(a_t\rot_\theta g_1v_1, a_t\rot_\theta g_2v_2)}\leq e^{-A_1\rho t}.
\]
Since $\darot\in\SO(\sqf)$, we conclude from the above that 
\be\label{eq: gv1 v2 are almost parallel}
\sqf(g_1v_1, g_2v_2)\leq  e^{-A_1\rho t}.
\ee
We claim: 
\be\label{eq: v1 is not too short}
\|g_1v_1\|\geq e^{\rho t}. 
\ee
Indeed if $\|g_1v_1\|< e^{\rho t}$, then since $\|g_2v_2\|\leq \|g_1v_1\|$, it follows  from~\eqref{eq: gv1 v2 are almost parallel} that 
$\{(g_1v_1,0), (0,g_2v_2)\}$ spans a $(\rho, A_1, t)$-special subspace. This contradicts the definition of $\tilde\omega$ and establishes~\eqref{eq: v1 is not too short}. 

\medskip 

Let us now return to estimating~\eqref{eq: L2 + epsilon : Theta-t}; we will estimate the sum on the right side of~\eqref{eq: L2 + epsilon : Theta-t} using the following elementary fact.

\begin{sublemma}
Let $t>0$, and let $w\in\R^2$ be a non-zero vector. Then  
\[
\ave\|a_t\rot_\theta w\|^{-2-2\delta}\diff\!\theta\leq \hat Ce^{4\delta t}\|w\|^{-2-2\delta}
\]
where $\hat C$ is absolute. 
\end{sublemma}

First note that~\eqref{eq: gv1 v2 are almost parallel} and the fact that part~1 in Lemma~\ref{lem: not many solutions} holds 
imply that exist $t_0$ and $C$ so that for all $t_0\leq n\leq t_1$, we have  
\be\label{eq: number of tilde Delta n}
\#\tilde\pcal_n\leq C e^{(2-\rho)n}.
\ee 
Also recall from~\eqref{eq: v1 is not too short} that $\norm{g_1v_1}\geq e^{\rho t}$, which in particular implies that $\norm{v_1}\gg e^{\rho t}$. Since $v_1\in\pcal_n$, we conclude that $n\geq \rho t+O(1)$. 
Thus~\eqref{eq: number of tilde Delta n} and the Sublemma imply that  
\be\label{eq: most vectors}
\begin{aligned}
\sum_{v_1\in\tilde\pcal_n}\int_{I_{v_1}}\|a_t\rot_\theta g_1v_1\|^{-2-2\delta}\diff\!\theta&\ll e^{(2-\rho)n} e^{4\delta t}e^{(-2-2\delta)n}\\
&\ll e^{-\rho^2 t} e^{4\delta t}\leq e^{-2\delta t}
\end{aligned}
\ee
in the last inequality, we used $\rho^2=A_1\delta/2\geq 100\delta$ and assumed $t$ is large.  

We now sum over all $n\leq t_1$ and get that 
\[
\sum_n\sum_{\tilde\pcal_n} \int_{I_{v_1}} \|a_t\rot_\theta g_1v_1\|^{-2-2\delta}\ll te^{-2\delta t}\ll e^{-\delta t}.
\]
This and~\eqref{eq: L2 + epsilon : Theta-t} complete the proof in this case. 

In combination with \eqref{eq: contribution of C'} and~\eqref{eq: contribution of Bt}, the proof is complete. 
\qed

\begin{proof}[Proof of the Sublemma]
Without loss of generality, we may assume $w=(0,1)$. Put 
\[
I=\Bigl[e^{(-2+2\delta)t}, 2\pi-e^{(-2+2\delta)t}\Bigr]\qquad\text{and}\qquad I'=[0,2\pi]\setminus I.
\]
Then 
\begin{align*}
\ave \frac{\diff\!\theta}{\|a_t\rot_\theta w\|^{2+2\delta}}&\ll \int_{I'}\frac{\diff\!\theta}{\|a_t\rot_\theta w\|^{2+2\delta}}+
\int_I\frac{\diff\!\theta}{\|a_t\rot_\theta w\|^{2+2\delta}}\\
&\ll e^{(-2+2\delta)t} e^{(2+2\delta)t}+\int_I\frac{\diff\!\theta}{\|a_t\rot_\theta w\|^{2+2\delta}}\\
&\leq e^{4\delta t}+\int_I\frac{\diff\!\theta}{\|a_t\rot_\theta w\|^{2+2\delta}}.
\end{align*}
We now compute the integral over $I$. Note that $\|a_t\rot_\theta w\|^{2+2\delta}\gg e^{(2+2\delta)t}\theta^{2+2\theta}$. Therefore, 
\begin{align*}
\int_I\frac{\diff\!\theta}{\|a_t\rot_\theta w\|^{2+2\delta}}&\ll e^{(-2-2\delta)t}\int_I\theta^{-2-2\delta}\diff\!\theta\\
&\ll e^{(-2-2\delta)t}e^{(1+2\delta)(2-2\delta)t}\ll e^{-4\delta^2 t}.
\end{align*}
The proof is complete. 
\end{proof}

We end this section with the proof of Lemma~\ref{lem: gen bd for int hat f}.

\begin{proof}[Proof of Lemma~\ref{lem: gen bd for int hat f}]
We begin with part~(1). Recall that $f_i$ is the characteristic function of $\{w\in\R^2: \norm{w}\leq R\}$, and let $f=f_1f_2$. 
Again by a variant of Schmidt's Lemma, we have 
\[
\hat f(\darot g\Gamma')\leq \omega_1(g_1\SL_2(\Z))\omega_2(g_2\SL_2(\Z))
\]
Let $\delta=\eta/10$. As it was done in~\eqref{eq: def Theta-t}, define 
\[
\Theta_t(\delta)\!=\!\Bigl\{\theta\in[0,2\pi]\!:\! \omega_2(a_t\rot_\theta x_2)^{1-2\delta}\leq \omega_1(a_t\rot_\theta x_1)\leq \omega_2(a_t\rot_\theta x_2)^{1+2\delta}\Bigr\}
\]
and let $\Theta_t'(\delta)=[0,2\pi]\setminus\Theta_t(\delta)$ where $x_i=g_i\SL_2(\Z)$.  
Then by Lemma~\ref{lem: alpha1 and alpha3}, we have for all $t\geq t(\delta)$ 
\be\label{eq: upp bd int f hat Theta'}
\int_{\Theta'_t(\delta)} \hat f(\darot g\Gamma')\diff\!\theta\leq \int_{\Theta'_t(\delta)} \Bigl(\omega_1(a_t\rot_\theta x_1)\omega_2(a_t\rot_\theta x_2)\Bigr)\diff\!\theta\ll1
\ee
the implied constant depends polynomially on the injectivity radius of $g\Gamma'$.
 
We now find an upper bound for the integral over $\Theta_t(\delta)$:
\[
\int_{\Theta_t(\delta)} \hat f(\darot g\Gamma')\diff\!\theta\leq \int \omega_1(a_tr_\theta x_1)^{2+2\delta}\diff\!\theta
\]
This, the sublemma, and standard arguments (which simplify significantly thanks to~\eqref{eq: only one short vector}), see e.g.~\cite{EskinMozes-MF}, imply that  
\[
\int_{\Theta_t(\delta)} \hat f(\darot g\Gamma')\diff\!\theta\ll e^{4\delta t}
\]
The claim in part~(1) of the lemma follows. 

We now turn to the proof of part~(2). 
Let $(v_1,0)$ and $(0,v_2)$ be as in the statement. 
For $i=1,2$ let $w_i=g_iv_i$. By a variant of Schmidt's Lemma,
\be\label{eq: upp bd int f hat schmidt}
\hat f(\theta)\leq \norm{a_t r_\theta w_1}^{-1} \norm{a_t r_\theta w_2}^{-1}. 
\ee 
For $i=1,2$, set  
\[
I_i=\Bigl\{\theta: R^{-1}e^{-\eta t}/10\leq \norm{a_t r_\theta w_i}\Bigr\}
\]
If $\theta\not\in I_1\cap I_2$, then $\hat f(\theta)>e^{\eta t}$. 
This,~\eqref{eq: upp bd int f hat schmidt}, and the definition of $\mathcal C_L$ imply 
\[ 
\int_{\mathcal C_L}\hat f(\theta)\leq \int_{I_1\cap I_2} \frac{1}{\norm{a_t r_\theta w_1}\norm{a_t r_\theta w_2}}.
\]
Thus, using Cauchy-Schwarz inequality, we need to find an upper bound for 
\[ 
\biggl(\int_{I_1} \frac{\diff\!\theta}{\norm{a_t r_\theta w_1}^{2}}\biggr)^{1/2} \biggl(\int_{I_2} \frac{\diff\!\theta}{\norm{a_t r_\theta w_2}^{2}}\biggr)^{1/2}.
\]

The computation is similar to the one in the proof of the sublemma. 
Indeed, we may assume $w_i=(0,1)$; then there is $R^{-1}\ll c<1$ so that 
\[
I_i\subset [ce^{-(1+\eta)t}, 2\pi-ce^{-(1+\eta)t}].
\] 
From this, we conclude that 
\[
\int_{I_i} \frac{\diff\!\theta}{\norm{a_t r_\theta w_i}^{2}}\ll e^{(-1+\eta)t},
\]
as it was claimed. 
\end{proof}


\section{Proof of Theorem~\ref{thm: equi of f hat}}\label{sec: proof of equi}
In this section, we will prove Theorem~\ref{thm: equi of f hat}. The proof combines a lower bound estimate, which will be proved using Theorem~\ref{thm: r effective equid}, with an upper bound estimate, which follows from Proposition~\ref{prop: main}, as we now explicate.

\begin{proof}[Proof of Theorem~\ref{thm: equi of f hat}]
Recall that $f_i\in C_c^\infty(\R^2)$, and $f$ is defined on $\R^4$ by $f(w_1,w_2)=f_1(w_1)f_2(w_2)$. We put
\be\label{eq: def hat f '}
\hat f(g'\Gamma')=\sum_{v\in g'\Lambda_{\rm nz}} f(v)
\ee
where  
$\Lambda = \{(v_1+v_2, \omega(v_1-v_2)): v_1,v_2\in \Z^2\}\subset \R^4$,
\[
\Gamma'=\{(\gamma_1,\gamma_2)\in\SL_2(\Z)\times\SL_2(\Z): \gamma_1 \equiv \omega \gamma_2 \omega \pmod{2}\}
\]
stabilizes $\Lambda$, and $g'=(g_1',g_2')\in G$. We also put $X=G/\Gamma'$.

Let $A$ and $\rho$ be as in the statement, and let $t>0$ be a parameter which is assumed to be large. Let $\hat A$ be a constant which will be explicated later, and let $g=(g_1,g_2)\in G$ satisfy the following:   
for every $Q\in\Mat_2(\Z)$ with $e^{\rho t/\hat A}\leq \norm{Q}\leq e^{\rho t}$ and all $\lambda\in\R$ we have  
\be\label{eq: far from rational proof stronger}
\norm{g_2^{-1}g_1-{\lambda}Q}>\norm{Q}^{-A/1000}.
\ee

We claim that~\eqref{eq: far from rational proof stronger} implies the following: 

\begin{sublemma} 
Let $g=(g_1, g_2)$ satisfy~\eqref{eq: far from rational proof stronger}.
There exists $A_1\geq \max(4D,A)$, where $D$ is as in Theorem~\ref{thm: r effective equid} so that the following holds. 
For all $t$ so that $t>4D\log t$ and for every $x\in X$ with $\vol(Hx)\leq e^{\rho t/A_1}$, we have 
\[
d(g\Gamma', x)> e^{-t/2}.
\]
\end{sublemma}

We first assume the sublemma and complete the proof of the theorem. In view of the sublemma, part~(1) in Theorem~\ref{thm: r effective equid} holds with $R=e^{\rho t/A_1}$ and $t$. Indeed, $D\rho/A_1\leq 1/4$ and $t^D\leq e^{t/4}$, which imply   
\[
R^Dt^D e^{-t}= e^{D\rho t/A_1} t^D e^{-t}\leq e^{-t/2};
\]
hence, part~(2) in Theorem~\ref{thm: r effective equid} cannot hold. 

For every $S$, let $1_{X_S}\leq \varphi_S\leq 1_{X_{S+1}}$ be a smooth function with $\Sob(\varphi_S)\ll S^\star$, where 
\[
X_\bullet=\{x=(x_1, x_2)\in X: \max(\omega_1(x_1), \omega_2(x_2))\leq \bullet\},
\]
see~\eqref{eq: def omegai} --- since $\Gamma'$ is a finite index subgroup of $\SL_2(\Z)\times\SL_2(\Z)$ this is well-defined.
Put $\hat f_S=\varphi_S\hat f$; we let $N$ be so that $\Sob(\hat f_S)\ll S^N\Sob(f)$. 

Put $\eta=\kappa_0\rho/(2NA_1)$, where $\kappa_0$ is as in Theorem~\ref{thm: r effective equid}. We will show the claim in the theorem holds with 
\[
\hat A= 3NAA_1/\kappa_0, \quad \delta_1=\eta,\quad\text{and}\quad \delta_2=\eta^3/A^3.
\]
First note that  
\be\label{eq: delta1 rho cond}
\rho/\hat A=\kappa_0\rho/(3NAA'_1)\leq \eta/A=\delta_1/A\leq \rho/100.
\ee

We now turn to the rest of the argument. Apply Lemma~\ref{lem: special subspace} with $(g_1, g_2)$ and the triple $(\eta/A,A, t)$. In view of~\eqref{eq: delta1 rho cond} and~\eqref{eq: far from rational proof stronger}, Lemma~\ref{lem: special subspace} implies that there are at most two $(\eta/A,A, t)$-special subspaces. 

Denote these subspaces by $L$ and $L'$ if they exist. For every $\theta\in[0,2\pi]$, we write 
\[
\hat f(\Delta(a_t\rot_\theta)g\Gamma')=\hat f_S(\Delta(a_t\rot_\theta)g\Gamma')+\hat f_{\rm cusp}(\Delta(a_t\rot_\theta)g\Gamma')+ \hat f_{\rm sp}(\Delta(a_t\rot_\theta)g\Gamma')
\]
where $\hat f_S=\varphi_S\hat f$, $\hat f_{\rm cusp}$ is the contribution of $g\Lambda_{\rm nz}\setminus(L\cup L')$ to $\hat f-\hat f_S$, and $\hat f_{\rm sp}$ is the contribution of $g\Lambda_{\rm nz}\cap (L\cup L')$ to $\hat f-\hat f_S$.

By Theorem~\ref{thm: r effective equid}, applied with $R=e^{\rho t/A'}$, 
for any smooth function $\xi$ on $[0,2\pi]$ we have 
\begin{multline}\label{eq: proof main term}
\biggl|\ave \hat f_S(\darot g\Gamma')\xi(\theta)\diff\!\theta- \ave\xi\diff\!\theta\int_X\hat f_S\diff\!m_X\biggr|\ll\\ 
\Sob(\hat f_S)\Sob(\xi)e^{-\kappa_0 \rho t/A'}\ll S^N\Sob(f) \Sob(\xi)e^{-\kappa_0 \rho t/A'}.
\end{multline}
If we choose $S=e^{\eta t}=e^{\kappa_0 \rho t/(2NA')}$, 
the above is $\ll \Sob(f)\Sob(\xi)e^{-\eta t/2}$. 

Moreover, by Lemma~\ref{lem: L2 bound SL2} applied with $p=3/2$ and the Chebyshev's inequality, we have 
\be\label{eq: lem L2 bound SL2 used}
\int_{\{\theta: \darot g\Gamma'\notin X_S\}} S \diff\!\theta\ll S^{-3/2} S=S^{-1/2}.
\ee
This and~\eqref{eq: proof main term}, reduce the problem to investigating the integral of $\hat f-\hat f_S=\hat f_{\rm cusp}+\hat f_{\rm sp}$ over 
$\hat{\mathcal C}:=\{\theta\in[0,2\pi]: \hat f-\hat f_S\geq S\}$.

Let $\tilde f$ be as in~\eqref{eq: def tilde f} with $\eta/A$, $A$, and $t$. That is:
\[
\tilde f(h;g\Gamma)=\sum_{v\in {\Nt}(g\Z^4)} f(\Delta(h)v)
\]
where ${\Nt}\Bigl(g\Z^4\Bigr)$ denotes the set of 
vectors in $g\Z^4$ not contained in any $(\eta/A,A, t)$-special subspaces and also not contained in $\R^2\times\{0\}\cup\{0\}\cup\R^2$. 

Let $\tilde{\mathcal C}_t=\{\theta\in[0,2\pi]: \tilde f(a_t\rot_\theta; g\Gamma)\geq 
e^{\eta t}=S\}$. 
By the definitions, 
\[
\int_{\hat{\mathcal C}}\hat f_{\rm cusp}(\Delta(a_t\rot_\theta)g\Gamma)\xi(\theta)\diff\!\theta\leq \norm{\xi}_{\infty}\int_{\tilde{\mathcal C}_t}\tilde f(a_t\rot_\theta;g\Gamma')\diff\!\theta.
\]
In view of~\eqref{eq: delta1 rho cond}, $e^{100\eta t/A}$ is in the range where~\eqref{eq: far from rational proof stronger} holds, thus Proposition~\ref{prop: main}, applied with $\eta/A$ and $A$, implies 
\[
\int_{\tilde{\mathcal C}_t}\tilde f(\Delta(a_t\rot_\theta)g\Gamma')\diff\!\theta\ll e^{-\eta^3 t/A^3}.
\]
From these two, we conclude that 
\be\label{eq: estimate for f hat cusp}
\int_{\hat{\mathcal C}}\hat f_{\rm cusp}(\Delta(a_t\rot_\theta) g\Gamma)\diff\!\theta\ll \norm{\xi}_\infty e^{-\eta^3 t/A^3}.
\ee

In view of~\eqref{eq: proof main term},~\eqref{eq: lem L2 bound SL2 used} and~\eqref{eq: estimate for f hat cusp}, we have 
\begin{multline*}
\biggl|\ave\hat f(\Delta(a_t\rot_\theta) g\Gamma)\xi(\theta)\diff\!\theta-\ave\xi\diff\!\theta\int_X\hat f_R\diff\!m_X\biggr|\\
= \int_{\mathcal C} \hat f_{\rm sp}(\Delta(a_t\rot_\theta) g\Gamma)\xi(\theta)\diff\!\theta+ O(\Sob(f)\Sob(\xi)e^{-\eta^2 t/A^3})
\end{multline*}
where $\mathcal C=\{\theta: \hat f_{\rm sp}(\Delta(a_t\rot_\theta)> e^{\eta t}\}$.

This completes the proof if we let $\delta_1=\eta$ and $\delta_2=\eta^3/A^3$.   
\end{proof}

\begin{proof}[Proof of the Sublemma]
Let $x=(h_1,h_2)\Gamma'$ be so that $Hx$ is periodic. 
In view of (the by now standard) non-divergence results, we may assume $\norm{h_i}\ll 1$ where the implied constant is absolute, see e.g.~\cite[\S3]{LM-PolyDensity}.

Since $\Gamma'$ is a finite index subgroup of $\SL_2(\Z)\times \SL_2(\Z)$, we conclude
\[
\{(h,h): h\in \SL_2(\R)\}\bigcap \Bigl(h_1\SL_2(\Z)h_1^{-1}\Bigr)\times \Bigl( h_2\SL_2(\Z)h_2^{-1}\Bigr)
\]
is a lattice in $\{(h,h): h\in \SL_2(\R)\}$. This implies that  $h_1\SL_2(\Z)h_1^{-1}$ and 
$h_2\SL_2(\Z)h_2^{-1}$ are commensurable. Hence, $h_2^{-1}h_1$ belongs to the image of $\GL_2^+(\Q)$ in $\SL_2(\R)$, i.e., the commensurator of $\SL_2(\Z)$ in $\SL_2(\R)$. 

Let $Q'\in\Mat_2(\Z)$ be so that $h_2^{-1}h_1={\lambda}Q'$,
where $\lambda=(\det Q')^{1/2}$. Since $\norm{h_i}\ll 1$, we have  
\be\label{eq: norm Q' vol Hx}
\norm{Q'}^{A_2}\ll \vol(Hx)\ll\norm{Q'}^{A_3},
\ee
where $A_2\leq 1\leq A_3$ and the implied constants are absolute, 
see e.g.~\cite[Lemma 16.2]{LMW22}.

We will show the sublemma holds with $A_1=4DA/A_2$. 
Assume now contrary to our claim in the sublemmsa that 
$\vol(Hx)\leq e^{\rho t/A_1}$, for some $A_1$ which will be determined later, and that $d_X(g\Gamma', x)\leq e^{-t/2}$.  

Thus $g_1=\epsilon_1 h_1\gamma_1$ and $g_2=\epsilon_2 h_2\gamma_2$ where $\norm{\epsilon_i}\ll e^{-t/2}$ and $(\gamma_1,\gamma_2)\in\Gamma'$. 
Since $\norm{h_i}\ll 1$, we conclude $\norm{\gamma_i}\ll \norm{g_i}$. Moreover, we have 
\be\label{eq: g2g1 h2h1}
g_2^{-1}g_1= \epsilon \gamma_2^{-1}h_2^{-1} h_1\gamma_1
\ee
where $\norm{\epsilon}\ll e^{-t/2}$ and the implied constants depend on $\norm{g_i}$. Put $Q=\gamma_2^{-1}Q'\gamma_1$. Then 
\[
\norm{Q}\ll\norm{Q'}\ll e^{\rho t/A_1A_2}\leq e^{\rho t/A}
\]
where we used~\eqref{eq: norm Q' vol Hx}, $\vol(Hx)\ll e^{\rho t/A_1}$ and assumed $t$ is large. Moreover, using~\eqref{eq: g2g1 h2h1} and~\eqref{eq: norm Q' vol Hx}, we conclude that 
\be\label{eq: g2g1 and h2h1 gives Q}
\norm{g_2^{-1}g_1-{\lambda}Q}\ll e^{-t/2}\norm{Q'}\ll e^{-t/2}\cdot e^{\rho t/ (A_1A_2)}
\ee
where the implied constants depend on $\norm{g_i}$. 

Assuming $t$ is large enough to account for the implied constant and using $A_1=4DA/A_2$, the left side of~\eqref{eq: g2g1 and h2h1 gives Q} is $<e^{-\rho t}$. Thus~\eqref{eq: g2g1 and h2h1 gives Q} contradicts~\eqref{eq: far from rational proof stronger} and finishes the proof of the theorem.   
\end{proof}

\bibliographystyle{halpha}
\bibliography{papers}

\end{document}